\documentclass[12pt]{article} 

\usepackage[top=2.5cm, bottom=2.5cm, left=2.5cm , right=2.5cm]{geometry} 
\usepackage[english]{babel}
\usepackage{hyperref}
\usepackage{amsmath,mathtools}
\usepackage{amsthm}
\usepackage{amssymb}
\usepackage{dsfont}
\usepackage{ucs}
\usepackage{enumerate}
\usepackage[utf8x]{inputenc}
\usepackage[T1]{fontenc}
\usepackage{xcolor}
\usepackage{tikz}
\usepackage{circuitikz}
\usepackage{setspace}
\newcommand{\R}{\mathbb R}
\newcommand{\N}{\mathbb N}
\newcommand{\Z}{\mathbb Z}

\newcommand{\C}{\mathbb C}

\renewcommand{\leqslant}{\leq}
\newcommand{\F}{\mathrm{\bold F}}

\newcommand{\LL}{\mathrm L}

\newcommand{\into}{\hookrightarrow}

\newcommand{\Aut}{\mathrm{Aut}}

\def\pmp{p.m.p.\@}
\newcommand{\defin}[1]{\textbf{\textit{#1}}}
\newcommand{\inv}{^{-1}}

\newcommand{\Stab}{\mathrm{Stab}}

\newcommand{\GGG}{\normalfont\textsf{G}}

\newcommand{\Cay}{\mathrm{Cay}}
\newcommand{\Isom}{\mathrm{Iso}}
\newcommand{\id}{\mathrm{id}}
\newcommand*\quot[2]{{^{\textstyle #1}\big/_{\textstyle #2}}}

\usepackage[draft,inline,nomargin,index]{fixme}

\fxsetup{theme=color,mode=multiuser}
\FXRegisterAuthor{mj}{amj}{\color{red}MJ}

\newtheorem{thm}{Theorem}[section]
\newtheorem{cor}[thm]{Corollary}

\newtheorem{lem}[thm]{Lemma}
\newtheorem{prop}[thm]{Proposition}

\theoremstyle{definition}
\newtheorem{claim}{Claim}
\newtheorem*{claim*}{Claim}
\newtheorem{ex}[thm]{Example}

\newenvironment{cproof}{\begin{proof}[Proof of the 
		claim]}{\end{proof}}

\newtheorem{qu}[thm]{Question}

\newtheorem{df}[thm]{Definition}
\newtheorem{rmq}[thm]{Remark}

\newtheorem*{convention}{Convention}

\title{Isometric orbit equivalence for probability-measure preserving actions}
\author{Matthieu Joseph}

\begin{document}
	
	\newcommand{\Addresses}{{
			\bigskip
			\footnotesize
			
			\noindent M.~Joseph, \textsc{Université de Lyon, ENS de Lyon, Unité de Mathématiques Pures et Appliquées, 46,
				allée d’Italie 69364 Lyon Cedex 07, FRANCE}\par\nopagebreak\noindent
			\textit{E-mail address: }\texttt{matthieu.joseph@ens-lyon.fr}
						
	}}

	\maketitle
	\abstract{We study probability-measure preserving (\pmp) actions of finitely generated groups via the graphings they define. We introduce and study the notion of isometric orbit equivalence for \pmp{} actions: two \pmp{} actions are isometric orbit equivalent if the graphings defined by some fixed generating systems of the groups are measurably isometric. 
	
	We highlight two kind of phenomena. First, we prove that the notion of isometric orbit equivalence is rigid for groups whose Cayley graph, with respect to a fixed generating system, has a countable group of automorphisms. On the other hand, we introduce a general construction of isometric orbit equivalent \pmp{} actions, which leads to interesting nontrivial examples of isometric orbit equivalent \pmp{} actions for the free group. In particular, we prove that mixing is not invariant under isometric orbit equivalence.
	
	In the appendix, we use our results to exhibit new concrete \pmp{} actions of free groups of finite rank that are mixing but not strongly ergodic, which bears witness that free groups have the Haagerup property.}
	\\
	
	\noindent\textbf{MSC:} 37A15, 37B05, 22F10. \\
	\textbf{Keywords:} Measure-preserving actions, graphings, Cayley graphs, orbit equivalence, quantitative orbit equivalence.
	\tableofcontents

\section{Introduction}

The purpose of this article is to compare probability-measure preserving actions of finitely generated groups via the graphings they define. A \defin{graphing} on a probability space $(X,\mu)$ is a graph $\mathcal{G}$ whose vertex set is $X$ and whose edge set is a measurable, symmetric subset of $X\times X$, such that the equivalence relation $\mathcal{R}$ "belonging to the same connected component of $\mathcal{G}$" satisfies the following two conditions: $(i)$ each $\mathcal{R}$-class is countable, $(ii)$ any bimeasurable bijection $T:A\to B$ between measurable subsets $A,B\subseteq X$, such that $(x,T(x))\in\mathcal{R}$ for all $x\in X$, preserves the measure $\mu$. 

Graphings have seen recently a sharp rise in interest as they play an important role in graph limit theory. Indeed, they serve as limit objects for sequences of bounded degree graphs, see \cite[Part~4]{lovasz} for an introduction to this theory. They are also one of the main objects in the cost theory of \pmp{} actions, which was extensively studied by Gaboriau \cite{gaboriaucout}. 
Graphings are closely related to probability measure preserving actions of countable groups. A probability measure preserving action (\pmp{} action for short) $\Gamma\curvearrowright^\alpha (X,\mu)$ of a countable group $\Gamma$ on a standard probability space $(X,\mu)$ is a collection $(\alpha(\gamma))_{\gamma\in\Gamma}$ of bimeasurable bijections $\alpha(\gamma):X\to X$ which preserve the probability measure $\mu$, such that for all $\gamma,\delta\in\Gamma$, we have $\alpha(\gamma\delta)=\alpha(\gamma)\alpha(\delta)$. A \pmp{} action is \defin{essentially free} if the set of points with trivial stabilizer has full measure. A \pmp{} action is \defin{ergodic} if any measurable set, which is invariant under the action, has measure $0$ or $1$. Assume that $\Gamma$ is finitely generated and fix $S$ a \defin{finite generating system} for $\Gamma$, that is, a finite symmetric set which generates the group and which does not contain the identity element $1_\Gamma$. To any \pmp{} action $\Gamma\curvearrowright^\alpha (X,\mu)$, one can associate a graphing, denoted by $\alpha(S_\Gamma)$, whose vertex set is $X$, and whose edge set is the symmetric set $\{(x,x')\in X\times X\colon \exists s\in S, \alpha(s)x=x'\}$.

The spirit of this article fits into the framework of \defin{quantitative orbit equivalence} and more generally \defin{quantitative measure equivalence} \cite{delabieQuantitativeMeasureEquivalence2020}. These nascent areas aim to understand how metric structures on orbits of \pmp{} actions are distorted under orbit and measure equivalences. 

Bounded orbit equivalence appears to be an important notion of quantitative orbit equivalence. For instance, in the entropy theory, Austin proved that any two \pmp{} essentially free actions of amenable groups that are bounded orbit equivalent (and even integrable orbit equivalent) have the same Kolmogorov-Sinai entropy \cite{austinBehaviourEntropyBounded2016}. Bowen and Lin proved that any two \pmp{} essentially free actions of a free group of finite rank, which are bounded orbit equivalent, have the same $f$-invariant \cite{bowenboundedOE}. Let $\Gamma$ and $\Lambda$ be two finitely generated groups and fix two finite generated systems $S_\Gamma$ and $S_\Lambda$ for $\Gamma$ and $\Lambda$ respectively. We say that two \pmp{} actions $\Gamma\curvearrowright^\alpha(X,\mu)$ and $\Lambda\curvearrowright^\beta(Y,\nu)$ are \defin{bounded orbit equivalent} if there exists a constant $C>0$ and a bimeasurable bijection $\Phi : X\to Y$ such that $\Phi_*\mu=\nu$ and for $\mu$-almost every $x\in X$, the map $\Phi$ is a $C$-biLipschitz map between the connected component of $x\in X$ in the graphing $\alpha(S_\Gamma)$ and the connected component of $\Phi(x)\in Y$ in the graphing $\beta(S_\Lambda)$. It is straightforward to check that bounded orbit equivalence does not depend on the choice of finite generating systems for the groups. 

This notion retains the geometry of the group. For instance, the following result can be extracted from Shalom's work \cite{shalomHarmonicAnalysisCohomology2004}: two finitely generated amenable groups $\Gamma$ and $\Lambda$ admit \pmp{} essentially free actions that are bounded orbit equivalent if and only if $\Gamma$ and $\Lambda$ are biLipschitz equivalent. This result is very specific to the amenable realm, as for instance free groups of different finite ranks are biLipschitz equivalent \cite{papasoglu} but not bounded orbit equivalent because they are not even orbit equivalent \cite{gaboriaucout}. While bounded orbit equivalence is related to the study of graphings which are uniformly biLipschitz, we propose here to study graphings associated with \pmp{} actions up to isometry.
\begin{df}
Let $\Gamma$ and $\Lambda$ be two finitely generated groups and fix two finite generating systems $S_\Gamma$ and $S_\Lambda$ respectively. Two \pmp{} actions $\Gamma\curvearrowright^\alpha(X,\mu)$ and $\Lambda\curvearrowright^\beta(Y,\nu)$ are \defin{isometric orbit equivalent} if the graphings $\alpha(S_\Gamma)$ and $\beta(S_\Lambda)$ are isometric in a measurable sense, that is, if there exists a bimeasurable bijection $\Phi : X\to Y$ such that $\Phi_*\mu=\nu$ and for $\mu$-almost every $x\in X$, the map $\Phi$ is an isometry between the connected component of $x\in X$ in the graphing $\alpha(S_\Gamma)$ and the connected component of $\Phi(x)\in Y$ in the graphing $\beta(S_\Lambda)$. 
\end{df}

Beware that contrary to bounded orbit equivalence, the notion of isometric orbit equivalence heavily depends on the choice of generating systems. 

\begin{convention}In each of the statements in this article, when we say that two \pmp{} actions are isometric orbit equivalent, it refers to the generating systems of the groups that are fixed and clear in the context.
\end{convention}

Consider $\Gamma$ a finitely generated group and fix $S_\Gamma$ a finite generating system of it. The \defin{Cayley graph} $(\Gamma,S_\Gamma)$ is the graph whose vertex set is $\Gamma$ and whose edge set is the symmetric set $\{(\gamma,\delta)\in\Gamma^2\colon\exists s\in S_\Gamma, \delta=\gamma s\}$. We endow the Cayley graph with its graph distance. Two Cayley graphs $(\Gamma,S_\Gamma)$ and $(\Lambda,S_\Lambda)$ are isometric if there exists a bijective isometry between them. Our first result shows that the existence of isometric orbit equivalent actions that are essentially free is connected to the isometry class of the Cayley graph. 
 
\begin{thm}[see Theorem \ref{thm.isomOEiffsameCayleygraphs}] Let $\Gamma$ and $\Lambda$ be two finitely generated groups and fix two finite generating systems $S_\Gamma$ and $S_\Lambda$ of $\Gamma$ and $\Lambda$ respectively. Then $\Gamma$ and $\Lambda$ admit \pmp{} essentially free actions that are isometric orbit equivalent if and only if the Cayley graphs $(\Gamma,S_\Gamma)$ and $(\Lambda,S_\Lambda)$ are isometric.
\end{thm}

One direction of this result is immediate, because in the graphing associated with any \pmp{} essentially free action, almost every connected component is isometric to the Cayley graph of the acting group. Thus, an isometric orbit equivalence between \pmp{} essentially free actions leads to an isometry between the Cayley graphs on almost every connected components of the associated graphings. The converse is proved using the space $\Isom(\Gamma,\Lambda)$ of bijective isometries from $(\Gamma,S_\Gamma)$ to $(\Lambda,S_\Lambda)$. The group $\Lambda$ acts by postcomposition on $\Isom(\Gamma,\Lambda)$. Similarly, the group $\Gamma$ acts by postcomposition on the space $\Isom(\Lambda,\Gamma)$ of bijective isometries from $\Lambda$ to $\Gamma$. We prove that the quotient actions $\Gamma\curvearrowright \Isom(\Gamma,\Lambda)/\Lambda$ and $\Lambda\curvearrowright\Isom(\Lambda,\Gamma)/\Gamma$, endowed with their respective Haar probability measures, are isometric orbit equivalent, see Corollary \ref{cor.naturalisomOE}.

This connection between Cayley graph and isometric orbit equivalence leads to interesting rigidity results. Given a finitely generated group $\Gamma$ with a generating system $S_\Gamma$, we denote by $\Isom(\Gamma)$ the group of bijective isometries of the Cayley graph $(\Gamma,S_\Gamma)$. Equivalently, $\Isom(\Gamma)$ is isomorphic to the group of all graph automorphisms of the Cayley graph. We prove the following result by using techniques introduced by Furman in \cite{furman}.

\begin{thm}[see Theorem \ref{thm.rigidity}]\label{thm.introrigidity}
Let $\Gamma$ and $\Lambda$ be two finitely generated groups and fix two finite generating systems $S_\Gamma$ and $S_\Lambda$ of $\Gamma$ and $\Lambda$ respectively. Assume that the Cayley graphs $(\Gamma,S_\Gamma)$ and $(\Lambda,S_\Lambda)$ are isometric. Let $\Gamma\curvearrowright^\alpha (X,\mu)$ and $\Lambda\curvearrowright^\beta (Y,\nu)$ be two \pmp{} essentially free actions that are isometric orbit equivalent. If $\Isom(\Gamma)$ (equivalently $\Isom(\Lambda)$) is countable, then 
\begin{enumerate}[(i)]
\item There exists finite index subgroups $\Gamma_0\leq\Gamma$ and $\Lambda_0\leq\Lambda$ which are isomorphic. 
\item There exists a $\Gamma_0$-invariant subset $X_0\leqslant X$ of positive measure and a $\Lambda_0$-invariant subset $Y_0\leqslant Y$ of positive measure such that the \pmp{} actions $\Gamma_0\curvearrowright (X_0,\mu_{X_0})$ and $\Lambda_0\curvearrowright (Y_0,\nu_{Y_0})$ are measurably isomorphic.
\end{enumerate}
If in addition, every finite index subgroup of $\Gamma$ acts ergodically on $(X,\mu)$, then $\Gamma$ and $\Lambda$ are isomorphic and $\alpha$ and $\beta$ are measurably isomorphic. 
\end{thm}

We discuss concrete examples of finitely generated groups $\Gamma=\langle S_\Gamma\rangle$ such that $\Isom(\Gamma)$ is countable in Example \ref{ex.fewisometries}. For instance, if $\Z^d$ is equipped with any finite generating set $S_{\Z^d}$, then the Cayley graph $(\Z^d,S_{\Z^d})$ has only countably many bijective isometries. The ergodicity constraint on any finite index subgroup is fulfilled for instance by Bernoulli shifts, or more generally by any mixing action. A \pmp{} action $\Gamma\curvearrowright^\alpha(X,\mu)$ is \defin{mixing} if for any measurable subset $A,B\subseteq X$, 
\[\mu(\alpha(\gamma)A\cap B)\underset{\gamma\to +\infty}{\longrightarrow}\mu(A)\mu(B).\]
The isometric orbit equivalence rigidity result obtained in Theorem \ref{thm.introrigidity} is false in the context of bounded orbit equivalence. Indeed, Fieldsteel and Friedman proved that for any $d\geq 2$ and any ergodic action $\Z^d\curvearrowright^\alpha (X,\mu)$, there exists a mixing \pmp{} action $\Z^d\curvearrowright^\beta (Y,\nu)$ such that $\alpha$ and $\beta$ are bounded orbit equivalent \cite[Thm.~3]{fieldsteelfriedman}. 

Let $\F_d$ be the free group on $d\geq 2$ generators $x_1,\dots,x_d$ with its generating set $S_{\F_d}=\{x_1^{\pm 1},\dots,x_d^{\pm 1}\}$. Then the Cayley graph $(\F_d,S_{\F_d})$ has uncountably many bijective isometries. We prove the following statement, which shows in particular that Theorem \ref{thm.introrigidity} is false in general if $\Isom(\Gamma)$ is uncountable. 

\begin{thm}\label{thm.introfreegroup} There exists ergodic actions $\F_d\curvearrowright^\alpha (X,\mu)$ and $\F_d\curvearrowright^\beta (Y,\nu)$ that are isometric orbit equivalent, such that $\alpha$ is not mixing but $\beta$ is.
\end{thm}

Therefore, mixing is not invariant under isometric orbit equivalence.  Theorem \ref{thm.introfreegroup} is false for $\F_1\simeq\Z$. In fact, for ergodic actions of $\Z$, Belinskaya proved in \cite{belinskayaPartitionsLebesguespace1968} that any two \pmp{} ergodic actions of $\Z$ that are $\LL^1$ orbit equivalent are flip-conjugate, that is, measurably isomorphic up to an automorphism of the group $\Z$. 

Theorem \ref{thm.introfreegroup} is proved using a general construction of isometric orbit equivalent actions that we explain now. For a concrete illustration of Theorem \ref{thm.introfreegroup}, we refer to Section \ref{sec.concreteexample}. 

Let $\Gamma$ be a finitely generated group and fix $S_\Gamma$ a finite generating system. Let $\Lambda\leq\Gamma$ be a finite index subgroup. The groups $\Lambda$ and $\Gamma$ both act by postcomposition on $\Isom(\Gamma)$. Then we prove in Section \ref{sec.generalconstruction} that the quotient action $\Gamma\curvearrowright\Isom(\Gamma)/\Lambda$,  equipped with its Haar probability measure, is isometric orbit equivalent to the diagonal action $\Gamma\curvearrowright\Isom(\Gamma)/\Gamma\times\Gamma/\Lambda$, equipped with the product of the Haar probability measure and the uniform probability measure. Since the diagonal action is never mixing, the aim in order to prove Theorem \ref{thm.introfreegroup} becomes to show that for some subgroups $\Lambda\leqslant\F_d$, the action $\F_d\curvearrowright\Isom(\F_d)/\Lambda$ is mixing. We actually give a complete characterization of such subgroups $\Lambda$ in Section \ref{sec.freegroupmixing}. 
Along the proof, we use the Howe-Moore property satisfied by the group $\Isom(\F_d)$, a result due independently to Lubotzky and Mozes \cite{lubotzky} and to Pemantle \cite{pemantle}.

In Appendix \ref{app}, we use the results of Section \ref{sec.freegroupmixing} to construct new concrete \pmp actions of $\F_d$ that are mixing but not strongly ergodic, which bears witness that free groups have the Haagerup property by a characterization due to Jolissaint \cite[Thm.~2.1.3]{Jolissaint}. 

\begin{thm} Let $d\geq 2$. Let $\mathbf{F}_d=\Gamma_0\leq\Gamma_1\leq\dots\leq\Gamma_n\leq\dots$ be a sequence of finite index subgroups ,each normal in $\mathbf{F}_d$, such that $(i)$ the sequence of Schreier graphs $\mathrm{Sch}(\mathbf{F}_d/\Gamma_n,S_{\mathbf{F}_d})$ is not a family of expanders and $(ii)$ none of the Schreier graphs $\mathrm{Sch}(\mathbf{F}_d/\Gamma_n,S_{\mathbf{F}_d})$ is bipartite. Let $m_n$ denotes the Haar probability measure on $\Isom(\F_d)/\Gamma_n$. Then the projective limit
\[\mathbb{F}_d\curvearrowright \varprojlim(\Isom(\F_d)/\Gamma_n,m_n)\]
is a \pmp{} mixing action that is not strongly ergodic. 
\end{thm}

\paragraph{Acknowledgments.} I would like to thank László Márton Tóth, which provides valuable feedback on this article. I thank Robin Tucker-Drob for suggesting me that Corollary \ref{cor.isomOEflexibility} could lead to a new proof that free groups have the Haagerup property.

\section{Preliminaries}
\subsection{Cayley graph}
In the text, if we write $\Gamma=\langle S\rangle$, we always mean that $\Gamma$ is a countable group and $S$ a \emph{\textbf{generating system}} for $\Gamma$, that is, a symmetric set which generates $\Gamma$, and which does not contain the identity element $1_\Gamma$ of $\Gamma$. In the case where the generating system $S$ is finite, then we say that $\Gamma$ is \defin{finitely generated}. When we say that $\Gamma=\langle S\rangle$ is finitely generated, we implicitly mean that $S$ is a finite generating system. 

Let $\Gamma=\langle S\rangle$ be a countable group. We denote by $\lvert\cdot\rvert_{S}$ the word length on $\Gamma$ with respect to the generating system $S$. It is defined by $\lvert 1_\Gamma\rvert_S =1$ and for all $\gamma\in\Gamma\setminus\{1_\Gamma\}$, 
\[\lvert \gamma\rvert_S\coloneqq \min\{n\in\N\colon \exists s_1,\dots,s_n\in S, \gamma=s_1\dots s_n\}.\]

The \emph{\textbf{Cayley graph}} of $\Gamma=\langle S\rangle$, denoted $(\Gamma,S)$, is the simplicial graph whose vertex set is $\Gamma$ endowed with the metric $(\gamma,\delta)\mapsto\lvert\gamma^{-1}\delta\rvert_S$, called the word metric. In practice, this means that $(\Gamma,S)$ is a graph without multiple edges or loops and there is an edge between two vertices $\gamma,\delta\in\Gamma$ if and only if there is $s\in S$ such that $\gamma=\delta s$. A map $f:\Gamma\to\Gamma$ is an isometry if for all $\gamma,\delta\in\Gamma$ we have
\[\lvert \gamma\inv\delta\rvert_S=\lvert f(\gamma)\inv f(\delta)\rvert_S.\]
The \defin{group of bijective isometries} of a Cayley graph $(\Gamma,S)$ is the group denoted by $\Isom(\Gamma,S)$ of all bijections $f:\Gamma\to\Gamma$ which are isometries. 
The action of $\Gamma$ by left multiplication on itself yields a canonical injective group homomorphism $\Gamma\into\Isom(\Gamma,S)$. Thus, $\Gamma$ identifies naturally as a subgroup of $\Isom(\Gamma,S)$. Each time the generating set is clear in the context, we will write $\Isom(\Gamma)$ instead of $\Isom(\Gamma,S)$. In the literature, the group $\Isom(\Gamma)$ is often referred to as the automorphism group of the Cayley graph $(\Gamma,S_\Gamma)$ and denoted by $\Aut(\Gamma,S_\Gamma)$.

Let $\Gamma=\langle S_\Gamma\rangle$ and $\Lambda=\langle S_\Lambda\rangle$ be two countable groups. A map $f:\Gamma\to\Lambda$ is an isometry if for all $\gamma,\delta\in\Gamma$ we have
\[\lvert\gamma^{-1}\delta\rvert_{S_\Gamma}=\lvert f(\gamma)^{-1}f(\delta)\rvert_{S_\Lambda}.\]
Let $\Isom((\Gamma,S_\Gamma),(\Lambda,S_\Lambda))$ be the space of bijective isometries between the Cayley graphs $(\Gamma,S_\Gamma)$ and $(\Lambda,S_\Lambda)$. We say that the Cayley graphs $(\Gamma,S_\Gamma)$ and $(\Lambda,S_\Lambda)$ are \defin{isometric} if $\Isom((\Gamma,S_\Gamma),(\Lambda,S_\Lambda))$ is nonempty. Again, if the generating systems are clear in the context, we will write $\Isom(\Gamma,\Lambda)$ instead of $\Isom((\Gamma,S_\Gamma),(\Lambda,S_\Lambda))$. 

\begin{ex} Here are some examples of finitely generated groups with isometric Cayley graphs.
\begin{enumerate}
\item Let $\Gamma=\Z$ be equipped with its usual generating system $S_\Gamma=\{-1,+1\}$. Let $\Lambda$ be the infinite dihedral group $D_\infty$, that is, the free product of two copies of the cyclic group $C_2$ of order two. A presentation of $D_\infty$ is given by 
\[D_\infty\coloneqq \langle a,b\mid a^2=b^2=1\rangle.\] 
Let $S_\Lambda$ be the usual generating system $\{a,b\}$. Then the Cayley graphs $(\Gamma,S_\Gamma)$ and $(\Lambda,S_\Lambda)$ are isometric. 

\item More generally, let $\Gamma=\F_d$ be the free group on $d\geq 2$ generators $x_1,\dots,x_d$, and let $S_{\Gamma}$ be the generating system $\{x_1^{\pm 1},\dots,x_d^{\pm 1}\}$. Let $W_{2d}$ be the universal Coxeter group of rank $2d$, that is, the free product of $2d$ copies of $C_2$. A presentation of $W_{2d}$ is given by 
\[W_{2d} = \langle a_1,\dots,a_{2d}\mid a_1^2=\dots=a_{2d}^2=1\rangle,\] 
and let $S_{W_{2d}}$ be the usual generating system $\{a_1,\dots,a_{2d}\}$. Then the Cayley graphs $(\Gamma,S_\Gamma)$ and $(W_{2d},S_{W_{2d}})$ are isometric. 

\item Let $g\geq 2$. Let $\Gamma_g$ be the fundamental group of a closed connected surface of genus $g$. A presentation of $\Gamma_g$ is given by
\[\Gamma_g=\langle x_1,y_1,\dots,x_g,y_g\mid [x_1,y_1]\dots[x_g,y_g]=1\rangle.\]
Let $S_{\Gamma_g}$ be the generating system $\{a_1^{\pm 1},b_1^{\pm 1},\dots,a_g^{\pm 1},b_g^{\pm 1}\}$. Let $\Lambda_g$ be the Coxeter group whose presentation is given by 
\[\Lambda_{g}=\langle a_1,\dots,a_{4g}\mid a_i^2=1, (a_ia_{i+1})^2=1, 1\leq i\leq 4g\rangle\] where the subscript in $(a_ia_{i+1})^2$ is taken modulo $4g$. Let $S_{\Lambda_g}$ be the generating system $\{a_1,\dots,a_{4g}\}$. Then the Cayley graphs $(\Gamma_g,S_{\Gamma_g})$ and $(\Lambda_g,S_{\Lambda_g})$ are isometric \cite{surfaceisometric}.
\item Let $m,n\geq 2$ be two integers. Let $\F_m$ and $\F_n$ be the free group on $m$ generators $x_1,\dots,x_m$ and on $n$ generators $y_1,\dots,y_n$ respectively. Let $\Gamma_{m,n}$ be the direct product $\F_m\times\F_n$ and $S_{\Gamma_{m,n}}$ be the generating system \[S_{\Gamma_{m,n}}\coloneqq\{(x_1^{\pm 1},1_{\F_n}),\dots,(x_m^{\pm 1},1_{\F_n}),(1_{\F_m},y_1^{\pm 1}),\dots,(1_{\F_m},y_n^{\pm 1})\}.\] It follows from the work of Burger and Mozes \cite{burgermozesTrees2000, burgermozesProduct2000} that for appropriate values of $m$ and $n$, the Cayley graph $(\Gamma_{m,n},S_{\Gamma_{m,n}})$ is isometric to the Cayley graph of some virtually simple group. 
\item Dyubina observed in \cite{dyubina} that there exist finitely generated groups $\Gamma=\langle S_\Gamma\rangle$ and $\Lambda=\langle S_\Lambda\rangle$ with isometric Cayley graphs, such that $\Gamma$ is solvable but $\Lambda$ not virtually solvable.  This example and the former one illustrate the fact that some algebraic properties are not preserved under the property of having isometric Cayley graphs.
\end{enumerate}
\end{ex}

\subsection{Definition of isometric orbit equivalence}


Let $(X,\mu)$ be a standard probability space. A bimeasurable bijection $T : X\to X$ is a \defin{\pmp{} automorphism} of $(X,\mu)$ if for all measurable set $A\subseteq X$, one has $\mu(T^{-1}(A))=\mu(A)$. We denote by $\Aut(X,\mu)$ the group of all \pmp{} automorphisms of $(X,\mu)$, two such automorphisms being identified if they coincide on a conull set. A bimeasurable bijection $T : (X,\mu)\to(Y,\nu)$ between two standard probability spaces is a \defin{\pmp{} isomorphism} if $T_*\mu=\nu$. A \pmp{} automorphism $T\in\Aut(X,\mu)$ is \defin{aperiodic} if the $T$-orbit of $\mu$-almost every $x\in X$ is infinite. A \pmp{} action of a countable group $\Gamma$ on $(X,\mu)$ is a homomorphism $\alpha : \Gamma\to\Aut(X,\mu)$. It is \defin{essentially free} if for each $\gamma\in\Gamma\setminus\{1_\Gamma\}$, the \pmp{} automorphism $\alpha(\gamma)$ is aperiodic. If there is no need to give a name to the \pmp{} action $\alpha$, we simply write $\gamma x$ instead of $\alpha(\gamma)x$. 

A \defin{\pmp{} partial automorphism} on $(X,\mu)$ is a bimeasurable bijection $T : A\to B$ between measurable subsets $A,B\subseteq X$, which preserves the measure, that is, for all measurable subset $C\subseteq A$, one has $\mu(T(C))=\mu(C)$. A \defin{graphing} on $(X,\mu)$ is a countable set $\Theta\coloneqq \{T_i : A_i\to B_i\mid i\in I\}$ of \pmp{} partial automorphisms. 
The shortest path distance $d_\Theta(x,y)$ between two distinct points $x,y\in X$ is the smallest integer $n\in\N\cup\{+\infty\}$ for which there exists $i_1,\dots,i_n\in I$ and $\varepsilon_1,\dots,\varepsilon_n\in \{\pm 1\}$ such that $y=T_{i_1}^{\varepsilon_1}\dots T_{i_n}^{\varepsilon_n}(x)$. This defines an (extended) metric $d_\Theta : X\times X\to\N\cup\{+\infty\}$. A measurable subset $A\subseteq X$ is \defin{$\Theta$-saturated} if $A$ is equal to the set of $y\in X$ such that there exists $x\in A$ with $d_\Theta(x,y)<+\infty$.

\begin{df}[Isometric graphings] Two graphings $\Theta$ on $(X,\mu)$ and $\Xi$ on $(Y,\nu)$ are \defin{measurably isometric} if there exist a \pmp{} isomorphism $\Phi : (X,\mu)\to(Y,\nu)$ and $\Theta$-saturated set $X_0\subseteq X$ with $\mu(X_0)=1$, such that $\Phi(X_0)$ is $\Xi$-saturated and  for all $x,x'\in X_0$ we have \[d_\Xi(\Phi(x),\Phi(x'))=d_\Theta(x,x').\]
\end{df}



\begin{df}[Isometric orbit equivalence]\label{df.isomOE}
Let $\Gamma=\langle S_\Gamma\rangle$ and $\Lambda=\langle S_\Lambda\rangle$ be two countable groups. Two \pmp{} actions $\Gamma\curvearrowright^\alpha (X,\mu)$ and $\Lambda\curvearrowright^\beta (Y,\nu)$ are \defin{isometric orbit equivalent} if the graphings $\alpha(S_\Gamma)$ on $(X,\mu)$ and $\beta(S_\Lambda)$ on $(Y,\nu)$ are measurably isometric. 
\end{df}

In the definition of isometric orbit equivalence, both groups $\Gamma$ and $\Lambda$ come with their fixed generating system. It is thus noteworthy that the notion of isometric orbit equivalence heavily depends on the generating systems of the groups.

\begin{ex} Let $T\in\Aut(X,\mu)$ be an aperiodic transformation. Let $C_2$ be the cyclic group of order two, and $\nu$ be the uniform probability measure on $C_2$. Let $\alpha$ be the \pmp{} $\Z$-action on $(X\times C_2,\mu\otimes\nu)$ defined by 
\[\alpha(n)(x,\varepsilon)\coloneqq\left\{\begin{array}{ll}(T^n(x),\varepsilon)& \text{if } n \text{ is even},\\ (T^n(x),1-\varepsilon)&\text{if } n \text{ is odd}.
\end{array}\right.\]
 Let $\beta$ be the \pmp{} action of the infinite dihedral group $D_\infty\coloneqq\langle a,b\mid a^2=b^2=1\rangle$ on $(X\times C_2,\mu\otimes u)$ defined by the action of the two generators
\begin{align*}
\beta(a)(x,0)\coloneqq(T(x),1) &,\quad \beta(a)(x,1)\coloneqq(T^{-1}(x),0),\\
\beta(b)(x,1)\coloneqq(T(x),0) &,\quad \beta(b)(x,0)\coloneqq(T^{-1}(x),1).
\end{align*}
Then the graphings $\alpha(\{\pm 1\})$ and $\beta(\{a,b\})$ are measurably isometric. Thus, if $\Z$ is equipped with the generating system $S_{\Z}\coloneqq\{\pm 1\}$ and $D_\infty$ with the generating system $S_{D_\infty}\coloneqq\{a,b\}$, then the \pmp{} actions $\alpha$ and $\beta$ are isometric orbit equivalent. 
\end{ex}

\begin{ex}\label{ex.coloring} Let $\Gamma=\langle S_\Gamma\rangle$ be a finitely generated group. Let $\Gamma\curvearrowright^\alpha(X,\mu)$ be a \pmp{} action and define \[E\coloneqq\{(x,x')\in X\times X\colon \exists s\in S_\Gamma, \alpha(s)x=x'\}.\] Fix a proper $d$-coloring  on $E$, that is, a measurable map $c: E\to \{1,\dots,d\}$ such that for $\mu$-almost every $x,x',x''\in X$, if $(x,x')$ and $(x,x'')$ are distinct elements of $E$, then $c(x,x')\neq c(x,x'')$. This always exists when $d$ is large enough (see \cite{CLP} for a precise statement) and the smallest such $d$ is called the measurable edge chromatic number of the graphing $\alpha(S_\Gamma)$. Let $W_d$ be the universal Coxeter group of rank $d$, which is given by the presentation
\[\Lambda\coloneqq \langle a_1,\dots,a_d\mid a_1^2=\dots=a_d^2=1\rangle\]
and let $S_{W_d}$ be the generating system $\{a_1,\dots,a_d\}$. Let $\beta$ be the \pmp{} action of $W_d$ defined by the action of its generators 
\[\beta(a_i)(x)\coloneqq\left\{\begin{array}{cl}x' &\text{ if }(x,x')\in E\text{ is such that }c(x,x')=i, \\ x &\text{ else.}\end{array}\right.\]
Then the \pmp{} actions $\alpha$ and $\beta$ are isometric orbit equivalent.
\end{ex}

We now explain a way to prove that \pmp{} actions are isometric orbit equivalent. For this, we need to introduce the notion of length-preserving cocycle. Let $\Gamma$ and $\Lambda$ be two countable groups. Let $\Gamma\curvearrowright (X,\mu)$ be a \pmp{} action. A measurable function $\sigma : \Gamma\times X\to\Lambda$ is a \defin{cocycle} if for all $\gamma,\delta\in\Gamma$,
\[\sigma(\gamma\delta,x)=\sigma(\gamma,\delta x)\sigma(\delta,x) \text{ for }\mu\text{-almost every }x\in X.\] 
If $\Gamma=\langle S_\Gamma\rangle$ and $\Lambda=\langle S_\Lambda\rangle$, we say that a cocycle $\sigma : \Gamma\times X\to\Lambda$ is \defin{length-preserving} if for all $\gamma\in\Gamma$, 
\[\lvert \sigma(\gamma,x)\rvert_{S_\Lambda}=\lvert\gamma\rvert_{S_\Gamma}\text{ for }\mu\text{-almost every }x\in X.\]

\begin{lem}\label{lem.isomOEnonfree} Let $\Gamma=\langle S_\Gamma\rangle$ and $\Lambda=\langle S_\Lambda\rangle$ be two countable groups. Let $\Gamma\curvearrowright^\alpha (X,\mu)$ and $\Lambda\curvearrowright^\beta (Y,\nu)$ be two \pmp{} actions. Assume that there is a \pmp{} isomorphism $\Phi : (X,\mu)\to (Y,\nu)$ and two length-preserving cocycles $\sigma : \Gamma\times X\to\Lambda$ and $\tau:\Lambda\times Y\to\Gamma$ such that for all $\gamma\in\Gamma$, $\lambda\in\Lambda$,
\begin{align*}
&\Phi(\alpha(\gamma)x)=\beta(\sigma(\gamma,x))\Phi(x),\quad\text{ for }\mu\text{-almost every }x\in X,\\
&\Phi(\alpha(\tau(\lambda,y))\Phi\inv(y))=\beta(\lambda)y,\quad\text{ for }\nu\text{-almost every }y\in Y.
\end{align*}
Then $\alpha$ and $\beta$ are isometric orbit equivalent.
\end{lem}
\begin{proof} Fix a $\Gamma$-invariant set $X_0\subseteq X$ of full measure such that $\Phi(X_0)$ is a $\Lambda$-invariant set and for all $\gamma\in\Gamma$, $\lambda\in\Lambda$, the equalities in the lemma hold for all $x\in X_0$ and $y\in \Phi(X_0)$. Then $X_0$ is $\alpha(S_\Gamma)$-saturated and $\Phi(X_0)$ is $\beta(S_\Lambda)$-saturated. Let us prove that for all $x,x'\in X_0$, we have $d_{\alpha(S_\Gamma)}(x,x')=d_{\beta(S_\Lambda)}(\Phi(x),\Phi(x'))$. Let $x,x'\in X_0$. First, we have $d_{\alpha(S_\Gamma)}(x,x')=+\infty$ if and only if $d_{\beta(S_\Lambda)}(\Phi(x),\Phi(x'))=+\infty$. Thus, we can assume that $d_{\alpha(S_\Gamma)}(x,x')<+\infty$. Let $\gamma\in\Gamma$ such that $d_{\alpha(S_\Gamma)}(x,x')=\lvert \gamma\rvert_{S_\Gamma}$ and $\alpha(\gamma)x=x'$. We have
\begin{align*}
d_{\beta(S_\Lambda)}(\Phi(x),\Phi(x')) &= d_{\beta(S_\Lambda)}(\Phi(x),\beta(\sigma(\gamma,x))\Phi(x)) \\
&\leq \lvert \sigma(\gamma,x)\rvert_{S_\Lambda} \\ 
&=\lvert \gamma\rvert_{S_\Gamma} \\
&=d_{\alpha(S_\Gamma)}(x,x'). 
\end{align*}
The reverse inequality is proved in a similar way, using the fact that $\tau$ is length-preserving. We thus obtain that \[d_{\beta(S_\Lambda)}(\Phi(x),\Phi(x'))=d_{\alpha(S_\Gamma)}(x,x'),\] which proves that $\alpha$ and $\beta$ are isometric orbit equivalent. 
\end{proof}

\subsection{Isometric orbit equivalence for essentially free actions}

Two \pmp{} actions of two countable groups $\Gamma\curvearrowright(X,\mu)$ and $\Lambda\curvearrowright(Y,\nu)$ are \defin{orbit equivalent} if there exists an \defin{orbit equivalence} between them, that is, a \pmp{} isomorphism $\Phi :(X,\mu)\to (Y,\nu)$ such that for $\mu$-almost every $x\in X$, 
\[\Phi(\Gamma x)=\Lambda\Phi(x).\]

Consider now two \pmp{} \emph{essentially free} actions $\Gamma\curvearrowright (X,\mu)$ and $\Lambda\curvearrowright (Y,\nu)$ and let $\Phi : (X,\mu)\to(Y,\nu)$ be an orbit equivalence between them. We therefore have two maps, which are uniquely defined (on a set of full measure) by freeness of the actions:
\begin{align*}
\sigma : \Gamma\times X\to\Lambda&,\quad \Phi(\gamma x)=\sigma(\gamma,x)\Phi(x),\\
\tau : \Lambda\times Y\to\Gamma&,\quad \Phi(\tau(\lambda,y)\Phi\inv(y))=\lambda y.
\end{align*}
They are called the \defin{orbit equivalence cocycles} associated with $\Phi$. Moreover, they satisfy the following properties.
\begin{lem}[Properties of the orbit equivalence cocycles]\label{lem.propertiescocycles} For $\mu$-almost every $x\in X$ and $\nu$-almost every $y\in Y$, the following hold.
\begin{enumerate}[(i)]
\item (Cocycles) The maps $\sigma : \Gamma\times X\to\Lambda$ and $\tau: \Lambda\times X\to\Gamma$ are cocycles. 
\item\label{item.bijection} (Bijections) The maps $\sigma(-,x):\Gamma\to\Lambda$ and $\tau(-,y):\Lambda\to\Gamma$ are bijections which are inverse one of each other: for all $\gamma\in\Gamma$, $\lambda\in\Lambda$,
\[\sigma(\tau(\lambda,y),\Phi\inv(y))=\lambda\text{ and }\tau(\sigma(\gamma,x),\Phi(x))=\gamma.\]
\item (Fixing the identity) We have $\sigma(1_\Gamma,x)=1_\Lambda$ and $\tau(1_\Lambda,y)=1_\Gamma$. 
\end{enumerate}
\end{lem}

\begin{rmq}\label{rmq.randomorphism} Let $[\Gamma,\Lambda]$ be the set of maps $f : \Gamma\to\Lambda$ such that $f(1_\Gamma)=f(1_\Lambda)$. There is a natural action $\Gamma\curvearrowright [\Gamma,\Lambda]$ defined by \[(\gamma\cdot f)(\delta)\coloneqq f(\gamma\inv)\inv f(\gamma\inv \delta).\] A \defin{randomorphism} from $\Gamma$ to $\Lambda$ is a $\Gamma$-invariant probability measure on $[\Gamma,\Lambda]$, see \cite[Def.~5.2]{monodInvitation}. Monod observed that randomorphisms supported on bijections can be obtained via orbit equivalences as follows. Let $\Phi$ be an orbit equivalence between \pmp{} essentially free actions $\Gamma\curvearrowright (X,\mu)$ and $\Lambda\curvearrowright (Y,\nu)$. Let $\sigma : \Gamma\times X\to \Lambda$ and $\tau : \Lambda\times Y\to\Gamma$ be the orbit equivalence cocycles associated with $\Phi$. For $x\in X$ and $y\in Y$, let $\sigma_x :\Gamma\to\Lambda$ and $\tau_y :\Lambda\to\Gamma$ be defined by  
\begin{align*}
\sigma_x(\gamma)&\coloneqq \sigma(\gamma\inv,x)\inv,\\
\tau_y(\lambda)&\coloneqq \tau(\lambda\inv,y)\inv.
\end{align*}
We deduce by Lemma \ref{lem.propertiescocycles} that $\sigma_x$ is a bijection such that $\sigma_x(1_\Gamma)=1_\Lambda$. Similarly, we get that $\tau_y$ is a bijection such that $\tau_y(1_\Lambda)=1_\Gamma$. Moreover, the cocycle identity for $\sigma$ and $\tau$ implies that $x\in X\mapsto \sigma_x\in [\Gamma,\Lambda]$ is $\Gamma$-equivariant and that $y\in Y\mapsto \tau_y\in [\Lambda,\Gamma]$ is $\Lambda$-equivariant. 
Thus, the pushforward of $\mu$ by $x\mapsto \sigma_x$ is a randomorphism from $\Gamma$ to $\Lambda$ supported on bijections. Similarly, the pushforward of $\nu$ by $y\mapsto \tau_y$ is a randomorphism from $\Lambda$ to $\Gamma$ supported on bijections.
\end{rmq}

Isometric orbit equivalent actions was defined in Definition \ref{df.isomOE} thanks to the graphings associated with the actions. We now characterize isometric orbit equivalent actions in terms of cocycles. In Lemma \ref{lem.isometriccocyclefree}, we will prove that two essentially free actions are isometric orbit equivalent if and only if there exists an orbit equivalence whose associated cocycles are length-preserving. 

\begin{df} Let $\Gamma=\langle S_\Gamma\rangle$ and $\Lambda=\langle S_\Lambda\rangle$ be two countable groups. Let $\Phi : (X,\mu)\to (Y,\nu)$ be an orbit equivalence between two \pmp{} essentially free actions $\Gamma\curvearrowright(X,\mu)$ and $\Lambda\curvearrowright(Y,\nu)$. We say that $\Phi$ is a \defin{length-preserving orbit equivalence} if the orbit equivalence cocycles $\sigma : \Gamma\times X\to\Lambda$ and $\tau : \Lambda\times X\to\Gamma$ associated with $\Phi$ are length-preserving.
\end{df}

\begin{rmq}\label{rmq.randisometry} Let $\Gamma=\langle S_\Gamma\rangle$ and $\Lambda=\langle S_\Lambda\rangle$ be two countable groups. Let $\Isom_1(\Gamma,\Lambda)$ be the set of $f\in\Isom(\Gamma,\Lambda)$ such that $f(1_\Gamma)=1_\Lambda$. This is a subspace of $[\Gamma,\Lambda]$ which is invariant under the action $\Gamma\curvearrowright\Isom_1(\Gamma,\Lambda)$ defined by \[(\gamma\cdot f)(\delta)\coloneqq f(\gamma\inv)\inv f(\gamma\inv\delta).\] Inspired by the language of randomorphism proposed by Monod \cite{monodInvitation}, we say that a $\Gamma$-invariant probability measure on $\Isom_1(\Gamma,\Lambda)$ is a \defin{randisometry} from $\Gamma$ to $\Lambda$. Randisometries can be obtained via length-preserving orbit equivalence as follows. Let $\Phi$ be a length-preserving orbit equivalence between two \pmp{} essentially free actions $\Gamma\curvearrowright (X,\mu)$ and $\Lambda\curvearrowright (Y,\nu)$. Let $\sigma : \Gamma\times X\to \Lambda$ and $\tau : \Lambda\times Y\to\Gamma$ be the orbit equivalence cocycles associated with $\Phi$. Let $\sigma_x\in [\Gamma,\Lambda]$ and $\tau_y\in [\Lambda,\Gamma]$ be defined as in Remark \ref{rmq.randomorphism}. Since $\Phi$ is length-preserving, we get by the cocycle identity that for all $\gamma,\delta\in\Gamma$,
\begin{align*}\lvert\sigma_x(\gamma)\inv\sigma_x(\delta)\rvert_{S_\Lambda}&=\lvert\sigma(\gamma\inv\delta,\delta\inv x)\rvert_{S_\Lambda} \\
&=\lvert \gamma\inv\delta\rvert_{S_\Gamma}.
\end{align*}
This means that $\sigma_x\in\Isom_1(\Gamma,\Lambda)$. Similarly, we have $\tau_y\in\Isom_1(\Gamma,\Lambda)$. Moreover, the map $x\in X\mapsto \sigma_x\in\Isom_1(\Gamma,\Lambda)$ is $\Gamma$-equivariant and the map $y\in Y\mapsto \tau_y\in\Isom_1(\Lambda,\Gamma)$ is $\Lambda$-equivariant. Thus, the pushforward of $\mu$ by $x\mapsto\sigma_x$ is a randisometry from $\Gamma$ to $\Lambda$. Similarly, the pushforward of $\nu$ by $y\mapsto \tau_y$ is a randisometry from $\Lambda$ to $\Gamma$. 
\end{rmq}

\begin{lem}\label{lem.isometriccocyclefree} Let $\Gamma=\langle S_\Gamma\rangle$ and $\Lambda=\langle S_\Lambda\rangle$ be two countable groups. Two \pmp{} essentially free actions $\Gamma\curvearrowright^\alpha (X,\mu)$ and $\Lambda\curvearrowright^\beta(Y,\nu)$ are isometric orbit equivalent if and only if there exists a length-preserving orbit equivalence between them. \end{lem}

\begin{proof}
If there exists a length-preserving orbit equivalence $\Phi$ between $\alpha$ and $\beta$, then we conclude by Lemma \ref{lem.isomOEnonfree} that $\alpha$ and $\beta$ are isometric orbit equivalent. Let us prove the converse. Assume that $\alpha$ and $\beta$ are isometric orbit equivalent. Let $\Phi : (X,\mu)\to(Y,\nu)$ be a pmp isomorphism, let $X_0\subseteq X$ be a set of full measure which is $\alpha(S_\Gamma)$-saturated, such that $\Phi(X_0)$ is $\beta(S_\Lambda)$-saturated and for all $x,x'\in X_0$, 
 \[d_{\alpha(S_\Gamma)}(x,x')=d_{\beta(S_\Lambda)}(\Phi(x),\Phi(x')).\]
This implies that $d_{\alpha(S_\Gamma)}(x,x')$ is finite if and only if $x$ and $x'$ are in the same $\Gamma$-orbit. Similarly, the distance $d_{\beta(S_\Lambda)}(\Phi(x),\Phi(x'))$ is finite if and only if $\Phi(x)$ and $\Phi(x')$ are in the same $\Lambda$-orbit. Thus, we deduce that $\Phi$ is an orbit equivalence. Let $\sigma : \Gamma\times X\to\Lambda$ and $\tau : \Lambda\times Y\to\Gamma$ be the orbit equivalence cocycles associated with $\Phi$. By definition, for all $\gamma\in\Gamma$ and all $x\in X_0$, we have
\begin{equation}\label{eq.distance}d_{\alpha(S_\Gamma)}(x,\alpha(\gamma)x)=d_{\beta(S_\Lambda)}(\Phi(x),\beta(\sigma(\gamma,x))\Phi(x)).\end{equation}
Since $\alpha$ and $\beta$ are essentially free, we can make the further assumption that $\alpha$ and $\beta$ are genuinely free on $X_0$ and $\Phi(X_0)$ respectively. Then the left hand side of \eqref{eq.distance} is equal to $\lvert \gamma\rvert_{S_\Gamma}$ while the right hand side of \eqref{eq.distance} is equal to $\lvert \sigma(\gamma,x)\rvert_{S_\Lambda}$. We thus get that $\lvert \sigma(\gamma,x)\rvert_{S_\Lambda}=\lvert \gamma\rvert_{S_\Gamma}$. Similarly, we get that $\lvert \tau(\lambda,y)\rvert_{S_\Gamma}=\lvert\lambda\rvert_{S_\Lambda}$ for all $\lambda\in\Lambda$ and  all $y\in \Phi(X_0)$. This proves that $\Phi$ is a length-preserving orbit equivalence. 
\end{proof}

\section{A canonical isometric orbit equivalence}\label{sec.canonical}

Let $\Gamma=\langle S_\Gamma\rangle$ and $\Lambda=\langle S_\Lambda\rangle$ be two finitely generated groups. The space $\Isom(\Gamma,\Lambda)$ of bijective isometries between the Cayley graphs $(\Gamma,S_\Gamma)$ and $(\Lambda,S_\Lambda)$ is a locally compact, totally disconnected space when equipped with the topology of pointwise convergence. The space $\Isom(\Gamma)$ of bijective isometries of the Cayley graph $(\Gamma,S_\Gamma)$ is a totally disconnected, locally compact group when endowed with the topology of pointwise convergence. Moreover, it contains naturally $\Gamma$ as a lattice, and thus is unimodular. 

Assume that $\Isom(\Gamma,\Lambda)$ is nonempty. Then the group $\Isom(\Gamma)$ acts simply transitively on $\Isom(\Gamma,\Lambda)$ by precomposition by the inverse. The group $\Isom(\Lambda)$ also acts simply transitively on $\Isom(\Gamma,\Lambda)$ by postcomposition. This implies that there exists a unique measure (up to a multiplicative constant) on $\Isom(\Gamma,\Lambda)$ which is invariant by the actions $\Isom(\Gamma)\curvearrowright\Isom(\Gamma,\Lambda)$ and $\Isom(\Lambda)\curvearrowright\Isom(\Gamma,\Lambda)$. We call it the Haar measure on $\Isom(\Gamma,\Lambda)$. The inverse map $\Isom(\Gamma,\Lambda)\to\Isom(\Lambda,\Gamma)$ is a bimeasurable bijection, which sends the Haar measure to the Haar measure. The pushforward of the Haar measure on $\Isom(\Gamma,\Lambda)$ by the quotient map $\Isom(\Gamma,\Lambda)\to\Isom(\Gamma,\Lambda)/\Lambda$, rescaled to have mass $1$, is called the Haar probability measure on $\Isom(\Gamma,\Lambda)/\Lambda$. Let $\Isom_1(\Gamma,\Lambda)$ be the compact open subspace of $\Isom(\Gamma,\Lambda)$ defined by 
\[\Isom_1(\Gamma,\Lambda)\coloneqq\{f\in\Isom(\Gamma,\Lambda)\colon f(1_\Gamma)=1_\Lambda\}.\]
This is a fundamental domain for the action $\Lambda\curvearrowright\Isom(\Gamma,\Lambda)$. Let $m$ be the Haar measure on $\Isom(\Gamma,\Lambda)$ such that $m(\Isom_1(\Gamma,\Lambda))=1$. The restriction of $m$ to $\Isom_1(\Gamma,\Lambda)$ is called the Haar probability measure on $\Isom_1(\Gamma,\Lambda)$. Let $m_\Lambda$ be the Haar probability measure on $\Isom(\Gamma,\Lambda)/\Lambda$. We then obtain that the \pmp{} action $\Isom(\Gamma,\Lambda)\curvearrowright (\Isom(\Gamma,\Lambda)/\Lambda,m_\Lambda)$ is measurably isomorphic to the \pmp{} action $\Isom(\Gamma,\Lambda)\curvearrowright(\Isom_1(\Gamma,\Lambda),m)$ defined by 
\[(g\cdot f)\coloneqq f(g\inv(1_\Gamma))\inv f\circ g\inv.\]
The restriction of this action to $\Gamma\leq\Isom(\Gamma,\Lambda)$ boils down to the \pmp{} action $\Gamma\curvearrowright (\Isom_1(\Gamma,\Lambda),m)$ encountered in Remark \ref{rmq.randisometry} and given by
\[(\gamma\cdot f)(\delta)\coloneqq f(\gamma\inv)\inv f(\gamma\inv\delta).\]
Let $m_\Gamma$ denote the Haar probability measure on $\Isom(\Lambda,\Gamma)/\Gamma$. One of the aim of this section is to prove that the actions $\Gamma\curvearrowright (\Isom(\Gamma,\Lambda)/\Lambda,m_\Lambda)$ and $\Lambda\curvearrowright (\Isom(\Lambda,\Gamma)/\Gamma,m_\Gamma)$ are isometric orbit equivalent. In order to prove this, we will work with the $\Gamma$-action on $\Isom_1(\Gamma,\Lambda)$ and the $\Lambda$-action on $\Isom_1(\Lambda,\Gamma)$ instead.

\begin{lem}\label{lem.naturalisomOE} Let $\Gamma=\langle S_\Gamma\rangle$ and $\Lambda=\langle S_\Lambda\rangle$ be two finitely generated groups, such that the Cayley graphs $(\Gamma,S_\Gamma)$ and $(\Lambda,S_\Lambda)$ are isometric. Let $\mu$ be the Haar probability measure on $\Isom_1(\Gamma,\Lambda)$ and $\nu$ the Haar probability measure on $\Isom_1(\Lambda,\Gamma)$. Then the following are true.
\begin{enumerate}[(i)]
\item\label{item.1} The map $\sigma : \Gamma\times\Isom_1(\Gamma,\Lambda)\to\Lambda$ defined by $\sigma(\gamma,f)\coloneqq f(\gamma\inv)\inv$ is a length-preserving cocycle.
\item\label{item.2} The map $\tau : \Lambda\times\Isom_1(\Lambda,\Gamma)\to\Gamma$ defined by $\tau(\lambda,f)\coloneqq f(\lambda\inv)\inv$ is a length-preserving cocycle.
\item\label{item.3} The inverse map $\Phi : (\Isom_1(\Gamma,\Lambda),\mu)\to(\Isom_1(\Lambda,\Gamma),\nu)$ is a \pmp{} isomorphism, such that for all $\gamma\in\Gamma,\lambda\in \Lambda$, we have
\begin{align*}
&\Phi(\gamma\cdot f)=\sigma(\gamma,f)\cdot\Phi(f),\quad \text{ for all }f\in\Isom_1(\Gamma,\Lambda), \\
&\Phi(\tau(\lambda,f)\cdot\Phi\inv(f))=\lambda\cdot f,\quad \text{ for all }f\in\Isom_1(\Lambda,\Gamma).
\end{align*}
\end{enumerate}
 \end{lem}
 
 \begin{proof}
We start by proving \eqref{item.1}. The fact that $\sigma$ is a cocycle is a straightforward computation:
\[\sigma(\gamma\delta,f)=f(\delta\inv\gamma\inv)\inv =f(\delta\inv\gamma\inv)\inv f(\gamma\inv)f(\gamma\inv)\inv =\sigma(\gamma,\delta\cdot f)\sigma(\gamma,f).\]
Moreover, since $f\in\Isom_1(\Gamma,\Lambda)$, we get that
\[\lvert \sigma(\gamma,f)\rvert_{S_\Lambda}=\lvert f(\gamma\inv)\inv f(1_\Gamma)\rvert_{S_\Lambda} = \lvert\gamma\rvert_{S_\Gamma}.\]
This proves that $\sigma$ is length-preserving. The proof of \eqref{item.2} is identical. For the proof of \eqref{item.3}, it is clear that $\Phi$ is a bimeasurable map. Moreover, since the inverse map sends Haar measure to Haar measure, we obtain that $\Phi_*\mu=\nu$. Finally, the two formulas left to prove are straightforward computations. 
\end{proof}

We obtain the following result as a corollary.

\begin{cor}\label{cor.naturalisomOE}
Let $\Gamma=\langle S_\Gamma\rangle$ and $\Lambda=\langle S_\Lambda\rangle$ be two finitely generated groups, such that the Cayley graphs $(\Gamma,S_\Gamma)$ and $(\Lambda,S_\Lambda)$ are isometric. Let $m_\Lambda$ be the Haar probability measure on $\Isom(\Gamma,\Lambda)/\Lambda$ and $m_\Gamma$ the Haar probability measure on $\Isom(\Lambda,\Gamma)/\Gamma$. Then the \pmp{} actions $\Gamma\curvearrowright (\Isom(\Gamma,\Lambda)/\Lambda,m_\Lambda)$ and $\Lambda\curvearrowright (\Isom(\Lambda,\Gamma)/\Gamma,m_\Gamma)$ are isometric orbit equivalent.
\end{cor}

\begin{proof} Let $\sigma :\Gamma\times \Isom_1(\Gamma,\Lambda)\to\Lambda$ be defined by $\sigma(\gamma,f)\coloneqq f(\gamma\inv)\inv$ and $\tau : \Lambda\times\Isom_1(\Lambda,\Gamma)\to\Gamma$ be defined by $\tau(\lambda,f)\coloneqq f(\lambda\inv)\inv$. By Lemma \ref{lem.naturalisomOE}, these are length-preserving cocycles. Let $\mu$ and $\nu$ be the Haar probability measures on $\Isom_1(\Gamma,\Lambda)$ and $\Isom_1(\Lambda,\Gamma)$ respectively. Then by Lemma \ref{lem.isomOEnonfree} we get that the actions $\Gamma\curvearrowright (\Isom_1(\Gamma,\Lambda),\mu)$ and $\Lambda\curvearrowright (\Isom_1(\Lambda,\Gamma),\nu)$ are isometric orbit equivalent. Thus, the \pmp{} actions $\Gamma\curvearrowright (\Isom(\Gamma,\Lambda)/\Lambda,m_\Lambda)$ and $\Lambda\curvearrowright (\Isom(\Lambda,\Gamma)/\Gamma,m_\Gamma)$ are isometric orbit equivalent.
\end{proof}

In general, the action $\Gamma\curvearrowright (\Isom(\Gamma,\Lambda)/\Lambda,m_\Lambda)$ is not essentially free. However, a standard trick can be used to obtain \pmp{} essentially free actions that are isometric orbit equivalent.

\begin{thm}\label{thm.isomOEiffsameCayleygraphs} Let $\Gamma=\langle S_\Gamma\rangle$ and $\Lambda=\langle S_\Lambda\rangle$ be two finitely generated groups. Then $\Gamma$ and $\Lambda$ admit \pmp{} essentially free actions that are isometric orbit equivalent if and only if the Cayley graphs $(\Gamma,S_\Gamma)$ and $(\Lambda,S_\Lambda)$ are isometric. 
\end{thm}

\begin{proof} Assume that $\Gamma\curvearrowright (X,\mu)$ and $\Lambda\curvearrowright (Y,\nu)$ are \pmp{} essentially free actions that are isometric orbit equivalent. By Lemma \ref{lem.isometriccocyclefree}, there exists a length-preserving orbit equivalence $\Phi : (X,\mu)\to (Y,\nu)$. Let $\sigma: \Gamma\times X\to\Lambda$ and $\tau:\Lambda\times Y\to\Gamma$ be the orbit equivalence cocycles associated with $\Phi$. For $x\in X$, let $\sigma_x : \Gamma\to\Lambda$ be defined by $\sigma_x(\gamma)\coloneqq \sigma(\gamma\inv,x)\inv$. By Remark \ref{rmq.randisometry}, we obtain that $\sigma_x\in\Isom_1(\Gamma,\Lambda)$ for $\mu$-almost every $x\in X$. Thus, there exists a bijective isometry between the Cayley graphs $(\Gamma,S_\Gamma)$ and $(\Lambda,S_\Lambda)$. 

Conversely, assume that the Cayley graphs $(\Gamma,S_\Gamma)$ and $(\Lambda,S_\Lambda)$ are isometric. Let $\mu$ and $\nu$ be the Haar probability measures on $\Isom_1(\Gamma,\Lambda)$ and $\Isom_1(\Lambda,\Gamma)$ respectively. By Corollary \ref{cor.naturalisomOE}, the \pmp{} actions $\Gamma\curvearrowright (\Isom_1(\Gamma,\Lambda),m_\Gamma)$ and $\Lambda\curvearrowright (\Isom_1(\Lambda,\Gamma),\nu)$ are isometric orbit equivalent. If these actions are essentially free, then the proof is complete. Else, we fix two \pmp{} essentially free actions $\Gamma\curvearrowright (X,\mu_X)$ and $\Lambda\curvearrowright (Y,\mu_Y)$ and consider the \pmp{} actions. 
\begin{align*}
\Gamma\curvearrowright (\Isom_1(\Gamma,\Lambda)\times X\times Y,\mu\otimes\mu_X\otimes\mu_Y), \quad &\gamma(f,x,y)\coloneqq (\gamma\cdot f, \gamma x, f(\gamma\inv)\inv y), \\
\Lambda\curvearrowright (\Isom_1(\Lambda,\Gamma)\times X\times Y,\nu\otimes\mu_X\otimes\mu_Y),\quad &\lambda(f,x,y)\coloneqq (\lambda\cdot f, f(\lambda\inv)\inv x, \lambda y).\end{align*}
These actions are essentially free. Moreover, as a direct consequence of Lemma \ref{lem.naturalisomOE}, there is a length-preserving orbit equivalence between them, which implies by Lemma \ref{lem.isometriccocyclefree} that they are isometric orbit equivalent. 
\end{proof}

\begin{rmq} The trick used at the end of the proof for getting essentially free actions while staying (isometric) orbit equivalent is due to Gaboriau \cite[Thm.~2.3]{gaboriausurvey}.
\end{rmq}

\section{Rigidity of isometric orbit equivalence}

The aim of this section is to understand isometric orbit equivalence when the space of bijective isometries $\Isom(\Gamma,\Lambda)$ between $\Gamma=\langle S_\Gamma\rangle$ and $\Lambda=\langle S_\Lambda\rangle$ is countable. Observe that the cardinality of $\Isom(\Gamma)$, $\Isom(\Lambda)$ and $\Isom(\Gamma,\Lambda)$ coincide, because the groups $\Isom(\Gamma)$ and $\Isom(\Lambda)$ acts simply transitively on $\Isom(\Gamma,\Lambda)$.

We say that two countable groups $\Gamma$ and $\Lambda$ are \defin{virtually isomorphic} if there exists finite index subgroups $\Gamma_0\leq\Gamma$ and $\Lambda_0\leq\Lambda$ which are isomorphic. We say that two \pmp{} actions $\Gamma\curvearrowright (X,\mu)$ and $\Lambda\curvearrowright (Y,\nu)$ are \defin{virtually measurably isomorphic} if there exist finite index subgroups $\Gamma_0\leq\Gamma$ and $\Lambda_0\leq\Lambda$, as well as a $\Gamma_0$-invariant subset $X_0\subseteq X$ of positive measure and a $\Lambda_0$-invariant subset $Y_0\subseteq Y$ of positive measure, such that the \pmp{} actions $\Gamma_0\curvearrowright (X_0,\mu_{X_0})$ and $\Lambda_0\curvearrowright (Y_0,\nu_{Y_0})$ are measurably isomorphic.

We prove a rigidity result for isometric orbit equivalence. The strategy of the proof is modeled on the proof of orbit equivalence rigidity phenomena due to Furman \cite{furman}.

\begin{thm}\label{thm.rigidity} Let $\Gamma=\langle S_\Gamma\rangle$ and $\Lambda=\langle S_\Lambda\rangle$ be two finitely generated groups, such that the Cayley graphs $(\Gamma,S_\Gamma)$ and $(\Lambda,S_\Lambda)$ are isometric. Let $\Gamma\curvearrowright^\alpha (X,\mu)$ and $\Lambda\curvearrowright^\beta (Y,\nu)$ be two \pmp{} essentially free actions that are isometric orbit equivalent. Assume that $\Isom(\Gamma)$ (equivalently $\Isom(\Lambda)$) is countable. Then $\Gamma$ and $\Lambda$ are virtually isomorphic groups and the \pmp{} actions $\alpha$ and $\beta$ are virtually measurably isomorphic. If in addition, every finite index subgroup of $\Gamma$ acts ergodically on $(X,\mu)$, then $\Gamma$ and $\Lambda$ are isomorphic and $\alpha$ and $\beta$ are measurably isomorphic. 
\end{thm}


\begin{proof} Thanks to Lemma \ref{lem.isometriccocyclefree}, we fix a length-preserving orbit equivalence $\Phi : (X,\mu)\to(Y,\nu)$ between $\alpha$ and $\beta$. Let $\sigma : X\times\Gamma\to\Lambda$ and $\tau : \Lambda\times Y\to\Gamma$ be the orbit equivalence cocycles associated with $\Phi$. For $x\in X$ and $y\in Y$, let $\sigma_x : \Gamma\to\Lambda$ and $\tau_y : \Lambda\to\Gamma$ be the maps defined by $\sigma_x(\gamma)\coloneqq\sigma(\gamma\inv,x)\inv$ and $\tau_y(\lambda)\coloneqq \tau(\lambda\inv,y)\inv$. By Remark \ref{rmq.randisometry}, we know that $\sigma_x\in\Isom_1(\Gamma,\Lambda)$ and $\tau_y\in\Isom_1(\Lambda,\Gamma)$ for $\mu$-almost every $x\in X$ and $\nu$-almost every $y\in Y$. Moreover, the map $x\mapsto \sigma_x$ is $\Gamma$-invariant and the map $y\mapsto\tau_y$ is $\Lambda$-invariant. Since $\Isom(\Gamma)$ is countable, the space $\Isom(\Gamma,\Lambda)$ is countable and thus the compact subset $\Isom_1(\Gamma,\Lambda)$ is finite. Thus there is $f_0\in\Isom_1(\Gamma,\Lambda)$ such that the set \[X_0\coloneqq \{x\in X\colon \sigma_x=f_0\}\] satisfies $\mu(X_0)>0$. We define \[\Gamma_0\coloneqq\{\gamma\in\Gamma\colon \gamma\cdot f_0=f_0\}.\] This is a finite index subgroup of $\Gamma$, which leaves $X_0$ invariant. Indeed, for all $\gamma\in\Gamma_0$ and $x\in X_0$, we have $\sigma_{\alpha(\gamma) x}=\gamma\cdot\sigma_x=\gamma\cdot f_0=f_0$. Let $g_0\in\Isom_1(\Lambda,\Gamma)$ be the inverse of $f_0$ and let \[Y_0\coloneqq\{y\in Y\colon \tau_y=g_0\}.\] We define \[\Lambda_0\coloneqq\{\lambda\in\Lambda\colon \lambda\cdot g_0=g_0\}.\] This is a finite index subgroup of $\Lambda$, which leaves $Y_0$ invariant. We know by Lemma \ref{lem.propertiescocycles}, that for $\mu$-almost every $x\in X$, for all $\gamma\in\Gamma$ and $\lambda\in\Lambda$,
\[\tau(\sigma(\gamma,x),\Phi(x))=\gamma\text{ and }\sigma(\tau(\lambda,\Phi(x)),x)=\lambda.\]
Thus, the maps $\sigma_x$ and $\tau_{\Phi(x)}$ are inverses of one another, that is $\sigma_x\circ\tau_{\Phi(x)}=\id_\Lambda$ and $\tau_{\Phi(x)}\circ\sigma_x=\id_{\Gamma}$. Therefore, we have $\Phi(X_0)=Y_0$. Thus the map $\Phi$ induces an orbit equivalence (still denoted by) $\Phi : (X_0,\mu_{X_0})\to (Y_0,\mu_{Y_0})$ between the \pmp{} actions $\Gamma_0\curvearrowright (X_0,\mu_{X_0})$ and $\Lambda_0\curvearrowright (Y_0,\nu_{Y_0})$. The orbit equivalence cocycles $\sigma_0: \Gamma_0\times X_0\to Y_0$ and $\tau_0 : \Lambda_0\times Y_0\to \Gamma_0$ associated with this orbit equivalence are independent of the space variable. Indeed for all $\gamma\in\Gamma_0$, $\lambda\in\Lambda_0$,
\begin{align*}
\sigma_0(\gamma,x)=\sigma(\gamma,x)=f_0(\gamma\inv)\inv \text{ for }\mu\text{-almost every }x\in X_0,\\
\tau_0(\lambda,y)=\tau(\lambda,y)=g_0(\lambda\inv)\inv \text{ for }\nu\text{-almost every }y\in Y_0.
\end{align*}
Thus, the groups $\Gamma_0$ and $\Lambda_0$ are isomorphic and the \pmp{} actions $\Gamma_0\curvearrowright (X_0,\mu_{X_0})$ and $\Lambda_0\curvearrowright (Y_0,\nu_{Y_0})$ are measurably isomorphic. This proves that the groups $\Gamma$ and $\Lambda$ are virtually isomorphic and that the actions $\alpha$ and $\beta$ are virtually isomorphic.

If in addition, every finite index subgroup of $\Gamma$ acts ergodically on $(X,\mu)$, then $\mu(X_0)=1$. Since $\Phi$ is a \pmp{} isomorphism, we deduce that $\nu(Y_0)=1$. In order to prove that $\alpha$ and $\beta$ are measurably isomorphic, it remains to show that $\Gamma_0=\Gamma$ and $\Lambda_0=\Lambda$. Up to null set, one can assume that $X_0$ is a $\Gamma$-invariant full measure set. Thus, for all $\gamma\in\Gamma$ and $x\in X_0$, we have $\gamma\cdot f_0=\gamma\cdot\sigma_x =\sigma_{\alpha(\gamma)x}=f_0$. Thus $\Gamma_0=\Gamma$. One proves similarly that $\Lambda_0=\Lambda$. We therefore conclude that $\Gamma$ and $\Lambda$ are isomorphic and that the \pmp{} actions $\alpha$ and $\beta$ are measurably isomorphic.
\end{proof}

Weakly mixing actions are examples of \pmp{} actions for which every finite index subgroup acts ergodically. Concrete examples of weakly mixing actions are Bernoulli shifts. Therefore, we have the following result.

\begin{cor}\label{cor.rigidity} Let $\Gamma=\langle S_\Gamma\rangle$ be a finitely generated group. Assume that $\Isom_1(\Gamma)$ is finite. Let $(A,\kappa)$ be a probability space. Any \pmp{} action  $\Lambda\curvearrowright^\beta (Y,\nu)$ of some finitely generated group $\Lambda=\langle S_\Lambda\rangle$ which is isometric orbit equivalent to the Bernoulli shift $\Gamma\curvearrowright (A,\kappa)^{\Gamma}$ is actually measurably isomorphic to it and $\Lambda$ is isomorphic to $\Gamma$.
\end{cor}

\begin{ex}\label{ex.fewisometries} Here are examples of finitely generated groups $\Gamma=\langle S\rangle$ such that $\Isom_1(\Gamma,S)$ is finite. Leemann and de la Salle proved that any finitely generated group $\Gamma$ admits a finite generating system $S$ such that $\Isom_1(\Gamma,S)$ is finite \cite{leemann}. 
For some finitely generated groups $\Gamma$, the set $\Isom_1(\Gamma,S)$ is finite for \emph{all} finite generating systems $S$. For instance, let $\Gamma$ be a finitely generated, torsion free group which is either of polynomial growth, or a lattice in a simple Lie group $G$ (in case $G\simeq\mathrm{SL}_2(\R)$, assume that $\Gamma$ is uniform in $G$). Then for any finite generating system $S$ of $\Gamma$, the space $\Isom_1(\Gamma,S)$ is finite. These facts are due to Trofimov for groups with polynomial growth \cite{trofimov} and to Furman for lattice in simple Lie groups \cite{furmanCayleyGraphs}. We refer to \cite[Sec.~6]{delasalle} for a discussion about these results. Other examples of such groups are obtained by Guirardel and Horbez. They proved the following result: if $\Gamma$ is a torsion-free finite index subgroup of the group of outer automorphisms of the free group $\F_d$ on $d\geq 3$ generators, then for any finite generating systems $S$ of $\Gamma$, the space $\Isom_1(\Gamma,S)$ is finite \cite{horbezguirardel}.
\end{ex}

\begin{rmq} The result of Corollary \ref{cor.rigidity} is false if $\Isom_1(\Gamma,\Lambda)$ is infinite. For instance, let $\Lambda_1$ and $\Lambda_2$ be two non-isomorphic finite groups. Let $\Gamma=\langle S_\Gamma\rangle$ be an infinite, finitely generated group and let $\Gamma_i\coloneqq \Lambda_i*\Gamma$ for $i\in\{1,2\}$, equipped with the finite generating set $\Lambda_i\cup S_\Gamma$. By a co-induction argument, one can show that the Bernoulli shifts $\Gamma_1\curvearrowright ([0,1],\mathrm{Leb})^{\Gamma_1}$ and $\Gamma_2\curvearrowright ([0,1],\mathrm{Leb})^{\Gamma_2}$ are isometric orbit equivalent, see for instance \cite[Thm.~1.1]{bowencoinduction}. However,  one can choose $\Lambda_1,\Lambda_2$ and $\Gamma$ so that $\Gamma_1$ and $\Gamma_2$ are not isomorphic.
\end{rmq}
\begin{qu} Let $\F_d$ be the free group on $d\geq 2$ generators $x_1,\dots,x_d$ and let $S_\Gamma\coloneqq \{x_1^{\pm 1},\dots,x_d^{\pm 1}\}$. Let $\Lambda=\langle S_\Lambda\rangle$ be a finitely generated group. If $\Lambda$ has a \pmp{} essentially free action $\Lambda\curvearrowright^\alpha (Y,\nu)$ which is isometric orbit equivalent to the Bernoulli shift $\F_d\curvearrowright^\beta ([0,1],\mathrm{Leb})^{\F_d}$, does this imply that $\Lambda$ is isomorphic to $\F_d$ and that the \pmp{} actions $\alpha$ and $\beta$ are measurably isomorphic?  
\end{qu}

This question is related to the following problem.  The measurable edge chromatic number of the graphing given by the Bernoulli shift $\F_d\curvearrowright ([0,1],\mathrm{Leb})^{\F_d}$ is known to be either $2d$ or $2d+1$ \cite{CLP}. However, its exact value is unknown \cite[Prob.~5.39]{kechrismarks}. As explained in Example \ref{ex.coloring}, the measurable edge chromatic number is equal to $2d$ if and only if the Bernoulli shift $\F_d\curvearrowright ([0,1],\mathrm{Leb})^{\F_d}$ is isometric orbit equivalent to some \pmp{} essentially free action of the universal Coxeter group of rank $2d$, which admits a presentation of the form
\[\langle a_1,\dots,a_{2d}\mid a_1^2=\dots=a_{2d}^2=1\rangle.\]

\begin{rmq} Let $\Gamma=\langle S_\Gamma\rangle$ be either a finitely generated amenable group, or the free group $\F_d$ on $d\geq 2$ generators with $S_\Gamma$ any free generating system. Then nontrivial Bernoulli shifts over $\Gamma$ are all orbit equivalent. This is a consequence of Ornstein and Weiss' theorem \cite{OrnsteinWeiss} if $\Gamma$ is amenable and a consequence of Bowen's theorem \cite{bowencoinduction} if $\Gamma$ is a free group. The picture is very different when it comes to isometric orbit equivalence. Let $(A,\kappa_A)$ and $(B,\kappa_B)$ be two nontrivial probability spaces. Then the Bernoulli shifts $\Gamma\curvearrowright (A,\kappa_A)^\Gamma$ and $\Gamma\curvearrowright (B,\kappa_B)^\Gamma$ are isometric orbit equivalent if and only if they are measurably isomorphic. This is a consequence of the fact that there is a notion of entropy which distinguishes Bernoulli shifts up to measure-isomorphism and which is preserved under bounded orbit equivalence and thus under isometric orbit equivalence. For amenable groups, Kolmogorov-Sinai entropy is preserved under bounded orbit equivalence \cite{austinBehaviourEntropyBounded2016}, whereas for free groups, the $f$-invariant is preserved under bounded orbit equivalence \cite{bowenboundedOE}.

 \end{rmq}

\section{Construction of isometric orbit equivalent actions}

\subsection{The general construction}\label{sec.generalconstruction}
Given a finitely generated group $\Gamma=\langle S_\Gamma\rangle$ and a finite index subgroup $\Lambda\leq\Gamma$, we explain in this section a construction of \pmp{} isometric orbit equivalent actions of $\Gamma$. In the sequel, we will use the following notation. For all $g\in\Isom(\Gamma)$ and all $f\in\Isom_1(\Gamma)$, we denote by $g\cdot f$ the element of $\Isom_1(\Gamma)$ defined by
\[g\cdot f :\delta\mapsto f(g\inv(1_\Gamma))\inv f(g\inv(\delta)).\] 
We explained in Section \ref{sec.canonical} why the action $(g,f)\mapsto g\cdot f$ is measurably isomorphic to the action $\Isom(\Gamma)\curvearrowright \Isom(\Gamma)/\Gamma$. Beware here, that $\Isom(\Gamma)/\Gamma$ means the quotient of $\Isom(\Gamma)$ under the $\Gamma$-action on $\Isom(\Gamma)$ by postcomposition.

\begin{lem}\label{lem.actionquotient}
Let $\Gamma=\langle S_\Gamma\rangle$ be a finitely generated group. Let $\Lambda\leq\Gamma$ be a finite index subgroup. Let $m_\Lambda$ be the Haar probability measure on $\Isom(\Gamma)/\Lambda$. Let $\mu$ be the Haar probability measure on $\Isom_1(\Gamma)$ and $u$ be the uniform probability measure on $\Gamma/\Lambda$. Then the \pmp{} action $\Isom(\Gamma)\curvearrowright (\Isom(\Gamma)/\Lambda,m_\Lambda)$ is measurably isomorphic to the \pmp{} action $\Isom(\Gamma)\curvearrowright (\Isom_1(\Gamma)\times\Gamma/\Lambda,\mu\otimes u)$ defined by
\[g(f,q)\coloneqq (g\cdot f,f(g\inv(1_\Gamma))\inv q).\]
\end{lem}


\begin{proof} We fix a section $s : \Gamma/\Lambda\to\Gamma$. For all $f\in\Isom_1(\Gamma)$ and $q\in\Gamma/\Lambda$, we let $\psi_{f,q}\in\Isom(\Gamma)$ be the map defined by $\psi_{f,q}(\gamma)\coloneqq s(q)f(\gamma)$. We define the subset $D\subseteq\Isom(\Gamma)$ by
\[D\coloneqq\{\psi_{f,q}\colon f\in\Isom_1(\Gamma), q\in \Gamma/\Lambda\}.\]
\begin{claim*} The set $D$ is a fundamental domain for $\Isom(\Gamma)/\Lambda$, that is, for all $f\in\Isom(\Gamma)$, there exists a unique $\lambda\in\Lambda$ such that the map $\gamma\mapsto \lambda\inv f(\gamma)$ belongs to $D$.  
\end{claim*}
\begin{cproof}
Let $f\in\Isom(\Gamma)$. Let $\delta=f(1_\Gamma)$. Then there exists $\lambda\in\Lambda$ and $q\in\Gamma/\Lambda$ such that $\delta=\lambda s(q)$. Observe that $\gamma\mapsto \delta\inv f(\gamma)$ belongs to $\Isom_1(\Gamma)$. Then the map $\gamma\mapsto \lambda\inv f(\gamma)$ belongs to $D$, because it coincides with $\psi_{\delta\inv f,q}$. 
\end{cproof}
This yields an action $\Isom(\Gamma)\curvearrowright D$, defined for all $g\in\Isom(\Gamma)$ and $\psi\in D$ by 
\[(g,\psi)\mapsto (\gamma\mapsto \lambda\inv \psi(g\inv(\gamma))),\]
where $\lambda$ is the unique element of $\Lambda$ such that the map $\gamma\mapsto \lambda\inv \psi(g\inv(\gamma))$ belongs to $D$. If we denote by $\mu_D$ the Haar measure on $\Isom(\Gamma)$ which satisfies $\mu_D(D)=1$, then the action $\Isom(\Gamma)\curvearrowright D$ preserves $\mu_D$ and is measurably isomorphic to the \pmp{} action $\Isom(\Gamma)\curvearrowright (\Isom(\Gamma)/\Lambda,m_\Lambda)$. 

Let us prove that the action $\Isom(\Gamma)\curvearrowright (D,\mu_D)$ is measurably isomorphic to the action $\Isom(\Gamma)\curvearrowright (\Isom_1(\Gamma)\times\Gamma/\Lambda,\mu\otimes u)$ defined for $g\in\Isom(\Gamma)$ and $(f,q)\in\Isom_1(\Gamma)\times\Gamma/\Lambda$ by
\[g(f,q)\coloneqq (g\cdot f,f(g\inv(1_\Gamma))\inv q).\]
Let us define a map $\Phi:\Isom_1(\Gamma)\times\Gamma/\Lambda\to D$ by the formula
\[\Phi(f,q)\coloneqq \gamma\mapsto \psi_{f,q}(\gamma).\]
This is a bijection, as it is a surjective map by definition of $D$ and it is straightforward to check that it is an injective map. Moreover, by definition of $\mu$ and $\mu_D$, we get $\Phi_*(\mu\otimes u)=\mu_D$. It is a straightforward computation to check that $\Phi$ intertwines the actions $\Isom(\Gamma)\curvearrowright \Isom_1(\Gamma)\times\Gamma/\Lambda$ and $\Isom(\Gamma)\curvearrowright D$, which finishes the proof of the lemma. 
\end{proof}

\begin{thm}\label{thm.constructionisomOE} Let $\Gamma=\langle S_\Gamma\rangle$ be a finitely generated group. Let $\Lambda\leq \Gamma$ be a finite index subgroup. Let $\mu_\Gamma$ be the Haar probability measure on $\Isom(\Gamma)/\Gamma$ and $u$ be the uniform probability measure on $\Gamma/\Lambda$. Then the \pmp{} action $\Gamma\curvearrowright (\Isom(\Gamma)/\Lambda,m_\Lambda)$ is isometric orbit equivalent to the diagonal action $\Gamma\curvearrowright (\Isom(\Gamma)/\Gamma\times \Gamma/\Lambda,m_\Gamma\otimes u)$. 
\end{thm}

\begin{proof} In this proof, we will denote by $*$ the \pmp{} action $\Gamma\curvearrowright (\Isom(\Gamma)/\Lambda,m_\Lambda)$ and by $\star$ the diagonal action $\Gamma\curvearrowright (\Isom(\Gamma)/\Gamma\times\Gamma/\Lambda,m_\Gamma\otimes u)$. 


Let $u$ be the uniform probability measure on $\Gamma/\Lambda$. We know by Lemma \ref{lem.actionquotient} that $*$ is measurably isomorphic to the action $\Gamma\curvearrowright (\Isom_1(\Gamma)\times\Gamma/\Lambda,\mu\otimes u)$, still denoted by $*$ and given by
\[\gamma*(f,q)\coloneqq (\gamma\cdot f,f(\gamma\inv)\inv q).\]
Let $\Phi : \Isom_1(\Gamma)\to\Isom_1(\Gamma)$ be the inverse map, defined by $f\circ \Phi(f)=\Phi(f)\circ f=\id_\Gamma$. This is a \pmp{} isomorphism by Lemma \ref{lem.naturalisomOE}. Morever, we have $\Phi(\gamma\cdot f)=f(\gamma\inv)\inv\Phi(f)$. 
Therefore, if $\Psi : \Isom_1(\Gamma)\times\Gamma/\Lambda\to \Isom_1(\Gamma)\times\Gamma/\Lambda$ is defined by $\Psi(f,q)=(\Phi(f),q)$, then $\Psi$ is a \pmp{} isomorphism which satisfies  
\begin{align*}
&\Psi(\gamma* (f,q))=f(\gamma\inv)\inv\star\Psi(f,q),\\
&\Psi(f(\gamma\inv)\inv *\Psi(f,q))=\gamma\star (f,q).
\end{align*}
By Lemma \ref{lem.naturalisomOE}, the map $(\gamma,f)\mapsto f(\gamma\inv)\inv$ is a length-preserving cocycle, thus we obtain by Lemma \ref{lem.isomOEnonfree} that the \pmp{} actions $*$ and $\star$ are isometric orbit equivalent, which concludes the proof. 
\end{proof}

In Section \ref{sec.concreteexample}, we give a concrete illustration of Theorem \ref{thm.constructionisomOE} with $\Gamma=\F_2$, the free group on two generators. 

\subsection{The case of the free group}\label{sec.freegroupmixing}

In this section, we characterize the subgroups $\Lambda\leq\F_d$ for which the \pmp{} action $\F_d\curvearrowright (\Isom(\F_d)/\Lambda,m_\Lambda)$ is mixing. Here, $m_\Lambda$ denotes the Haar probability measure. Before this, we need to give some properties of mixing actions of locally compact groups.

Let $G$ be a locally compact, non-compact, second countable group. A function $f:G\to\C$ vanishes at infinity if for all $\varepsilon >0$, the set $\{g\in G\colon \lvert f(g)\rvert \geq\varepsilon \}$ is compact. A \pmp{} action $G\curvearrowright (X,\mu)$ is \defin{mixing} if for all measurable subsets $A,B\subseteq X$, the function 
\[g\mapsto \lvert \mu(gA\cap B)-\mu(A)\mu(B)\rvert\]
vanishes at infinity. With this definition, the proofs of the following two lemmas are straightforward.

\begin{lem}\label{lem.subgroupmixing} Let $G\curvearrowright (X,\mu)$ be a \pmp{} mixing action. Then for any closed subgroup $H\leq G$, the action $H\curvearrowright (X,\mu)$ is mixing. 
\end{lem}

\begin{lem}\label{lem.fisubgroupmixing} Let $H\leq G$ be a finite index closed subgroup. Let $G\curvearrowright (X,\mu)$ be a \pmp{} action. If $H\curvearrowright (X,\mu)$ is mixing, then so is $G\curvearrowright (X,\mu)$.
\end{lem}

In this section, we fix an integer $d\geq 2$ and we let $\F_d$ be the free group on $d$ generators $x_1,\dots,x_d$ with the generating system $S\coloneqq\{a_1^{\pm 1},\dots,a_d^{\pm 1}\}$. The \defin{even subgroup} of $\F_d$ is the normal subgroup of index two, denoted by $\F_d^{ev}$, consisting of all elements $\gamma\in \F_d$ such that $\lvert\gamma\rvert_S$ is even. The \defin{even subgroup} of $\Isom(\F_d)$ is the closed, normal subgroup of index two of $\Isom(\F_d)$, denoted by $\Isom^{ev}(\F_d)$ and defined by

\[\Isom^{ev}(\F_d)\coloneqq\{f\in\Isom(\F_d)\colon f(\F_d^{ev})=\F_d^{ev}\}.\]
This group satisfies the Howe-Moore property.

\begin{thm}[Lubotzky-Mozes \cite{lubotzky}, Pemantle \cite{pemantle}]\label{thm.lubotzkymozes} Any \pmp{} ergodic action of $\Isom^{ev}(\F_d)$ on a standard probability space is mixing. 
\end{thm}

The following lemma characterizes the finite index subgroups of the even subgroup $\F_d^{ev}$.

\begin{lem}\label{lem.even} Let $\Lambda\leq \F_d$ be a finite index subgroup. Then the following are equivalent. 
\begin{enumerate}[(i)]
\item\label{item.notcontained} $\Lambda$ is not contained in $\F_d^{ev}$.
\item\label{item.index} For all $\gamma\in \F_d$, the index $[\F_d^{ev}:\gamma\Lambda\gamma\inv\cap\F_d^{ev}]$ is equal to $[\F_d:\Lambda]$. 
\item\label{item.transitive} The action $\F_d^{ev}\curvearrowright \F_d/\Lambda$ is transitive.
\item\label{item.bipartition} There is no bipartition $\F_d/\Lambda=U\sqcup V$ such that for all $s\in S$, $sU=V$. 
\end{enumerate}
\end{lem}

\begin{proof} Let us prove \eqref{item.notcontained}$\Rightarrow$\eqref{item.index}. Since $\Lambda$ is not contained in $\F_d^{ev}$, there is $\lambda\in\Lambda\setminus\F_d^{ev}$. Let $\gamma\in \F_d$. Since $\F_d^{ev}$ is normal in $\F_d$, the element $\gamma\lambda\gamma\inv$ is not in $\F_d^{ev}$. Since $\F_d^{ev}$ has index two in $\F_d$, we deduce that $\gamma\Lambda\gamma\inv\F_d^{ev}=\F_d$. Thus, we obtain
\[[\F_d:\gamma\Lambda\gamma\inv\cap\F_d^{ev}]=[\F_d:\gamma\Lambda\gamma\inv][\F_d:\F_d^{ev}].\]
We obtain \eqref{item.index} by dividing both sides of the equality by $[\F_d:\F_d^{ev}]$. 

We now prove \eqref{item.index}$\Rightarrow$\eqref{item.transitive}. Observe that for all $\gamma\in \F_d$, the group $\gamma\Lambda\gamma\inv\cap\F_d^{ev}$ is equal to the stabilizer of the coset $\gamma\Lambda$ under the action $\F_d^{ev}\curvearrowright \F_d/\Lambda$. Thus, the index $[\F_d^{ev}:\gamma\Lambda\gamma\inv\cap\F_d^{ev}]$ is equal to the cardinal of the orbit of the coset $\gamma\Lambda$ under the action $\F_d^{ev}\curvearrowright \F_d/\Lambda$. If we assume \eqref{item.index}, then we get that the cardinal of each orbit of the action $\F_d^{ev}\curvearrowright \F_d/\Lambda$ is equal to $[\F_d:\Lambda]$. This exactly means that the action is transitive. 

Let us prove the contrapositive of \eqref{item.transitive}$\Rightarrow$\eqref{item.bipartition}. Assume that there exists a partition $\F_d/\Lambda = U\sqcup V$ such that for all $s\in S$, $sU=V$. Then we also have $sV=U$ for all $s\in S$. By induction, we get that $\gamma U=U$ and $\gamma V=V$ for all $\gamma\in\F_d^{ev}$. Thus, the action $\F_d^{ev}\curvearrowright \F_d/\Lambda$ is not transitive.

Finally, let us prove the contrapositive of \eqref{item.bipartition}$\Rightarrow$\eqref{item.notcontained}. Assume that $\Lambda\leq\F_d^{ev}$. Let $\gamma\in \F_d\setminus \F_d^{ev}$. Let $n\coloneqq [\F_d^{ev}:\Lambda]$. Then there are $\gamma_1,\dots,\gamma_{2n}\in \F_d$ such that 
\[
\F_d^{ev}=\bigsqcup_{i=1}^n\gamma_i\Lambda\text{  and  }
\F_d\setminus\F_d^{ev}=\bigsqcup_{i=n+1}^{2n}\gamma_i\Lambda.
\]
But for all $s\in S$, we have $s\F_d^{ev}=\F_d\setminus\F_d^{ev}$. Thus, this decomposition of $\F_d^{ev}$ and $\F_d\setminus\F_d^{ev}$ into a disjoint union of $\Lambda$-coset yields a bipartition $\F_d/\Lambda = U\sqcup V$ such that for all $s\in S, sU=V$. 
\end{proof}

We can now characterize the subgroups $\Lambda\leqslant\F_d$ for which the \pmp{} action $\F_d\curvearrowright(\Isom(\F_d)/\Lambda,m_\Lambda)$ is mixing.

\begin{thm}\label{thm.notcontained} Let $\Lambda\leq\F_d$ be a finite index subgroup. Let $m_\Lambda$ be the Haar probability measure on $\Isom(\F_d)/\Lambda$. Then the following are equivalent.
\begin{enumerate}[(i)]
\item\label{item.grosmixing} The \pmp{} action $\Isom(\F_d)\curvearrowright (\Isom(\F_d)/\Lambda,m_\Lambda)$ is mixing. 
\item\label{item.petitmixing} The \pmp{} action $\F_d\curvearrowright (\Isom(\F_d)/\Lambda,m_\Lambda)$ is mixing. 
\item\label{item.notcontainedmixing} $\Lambda$ is not contained in the even subgroup $\F_d^{ev}$. 
\end{enumerate}
\end{thm}

\begin{proof} In this proof, we denote by $\mu$ the Haar probability measure on $\Isom_1(\F_d)$. Let $\cdot$ be the \pmp{} action $\Isom(\F_d)\curvearrowright(\Isom_1(\F_d),\mu)$ given by 
\[(g\cdot f)\coloneqq f(g\inv(1_{\F_2}))\inv f\circ g\inv.\]
 If $u$ denotes the uniform probability measure on $\F_d/\Lambda$, then by Lemma \ref{lem.actionquotient}, the \pmp{} action $\Isom(\F_d)\curvearrowright (\Isom(\F_d)/\Lambda,m_\Lambda)$ is measurably isomorphic to the \pmp{} action $\Isom(\F_d)\curvearrowright (\Isom_1(\F_d)\times \F_d/\Lambda,\mu\otimes u)$ given by \[g(f,q)\coloneqq (g\cdot f,f(g\inv(1_\Gamma))\inv q).\]

The proof of \eqref{item.grosmixing}$\Rightarrow$\eqref{item.petitmixing} is a direct consequence of Lemma \ref{lem.subgroupmixing}. 

Let us prove \eqref{item.petitmixing}$\Rightarrow$\eqref{item.notcontainedmixing}. We prove the contrapositive. Assume that $\Lambda$ is contained in $\F_d^{ev}$. Then by Lemma \ref{lem.even}, there is a bipartition $\F_d/\Lambda =U\sqcup V$ such that for all $s\in S$, we have $sU=V$. By induction, we get that $\gamma U=U$ and $\gamma V=V$ for all $\gamma\in\F_d^{ev}$. For all $f\in\Isom_1(\F_d)$ and all $\gamma\in \F_d$, we have $\lvert f(\gamma\inv)\inv\rvert_{S}=\lvert\gamma\rvert_S$. Thus we obtain that the sets $\Isom_1(\F_d)\times U$ and $\Isom_1(\F_d)\times V$ are invariant by $\F_d^{ev}$. Thus, the \pmp{} action $\F_d\curvearrowright (\Isom_1(\F_d)\times \F_d/\Lambda,\mu\otimes u)$ is not mixing, which is equivalent to saying that $\F_d\curvearrowright (\Isom(\F_d)/\Lambda,m_\Lambda)$ is not mixing.

We now prove \eqref{item.notcontainedmixing}$\Rightarrow$\eqref{item.grosmixing}. We prove the contrapositive. Assume that $\Isom(\F_d)\curvearrowright (\Isom(\F_d)/\Lambda,m_\Lambda)$ is not mixing. That is, the action $\Isom(\F_d)\curvearrowright (\Isom_1(\F_d)\times \F_d/\Lambda,\mu\otimes u)$ is not mixing. By Lemma \ref{lem.fisubgroupmixing}, the \pmp{} action $\Isom^{ev}(\F_d)\curvearrowright (\Isom_1(\F_d)\times \F_d/\Lambda,\mu\otimes u)$ is not mixing and thus not ergodic by Theorem \ref{thm.lubotzkymozes}. Let $A\subseteq\Isom_1(\F_d)\times \F_d/\Lambda$ be a measurable subset of measure $1/2$ which is $\Isom^{ev}(\F_d)$-invariant. We define the following set
\[U\coloneqq\{q\in\F_2/\Lambda\colon \mu\otimes u(A\cap (\Isom_1(\F_d)\times\{q\}))>0\}.\]
Observe that the subgroup $\Isom_1(\F_d)$ is contained in $\Isom^{ev}(\F_d)$. Moreover, for all $g\in\Isom_1(\F_d)$ and for all $(f,q)\in\Isom_1(\F_d)\times\F_d/\Lambda$, we have
\[g(f,q)=(f\circ g\inv,q).\]
Thus, the group $\Isom_1(\F_d)$ acts transitively on each $\Isom_1(\F_d)\times\{q\}$. Since $A$ is $\Isom^{ev}(\F_d)$-invariant, we obtain that $A=\Isom_1(\F_d)\times U$ up to a conull set. We claim that the set $U$ and its complement $V$ form a partition of $\F_d/\Lambda$ such that for all $s\in S, sU=V$. Indeed, for $s\in S$, the facts that $s\notin\Isom^{ev}(\F_d)$ and that $\Isom^{ev}(\F_d)$ is normal in $\Isom(\Gamma)$ imply that $sA$ is $\Isom^{ev}(\F_d)$-invariant. But $sA$ cannot be equal to $A$, because otherwise $A$ would be a $\Isom(\F_d)$-invariant set of measure $1/2$, contradicting the ergodicity of $\Isom(\F_d)\curvearrowright (\Isom(\F_d)/\Lambda,m_\Lambda)$. Moreover, the intersection $sA\cap A$ cannot be of positive measure, otherwise it would be a $\Isom^{ev}(\F_d)$-invariant set of measure $<1/2$, which is impossible since $\Isom^{ev}(\F_d)$ has index two in $\Isom(\F_d)$. Thus, up to conull set, we have $sA= (\Isom_1(\F_d)\times\F_d/\Lambda)\setminus A$. Therefore $sU=V$. By Lemma \ref{lem.even}, we obtain that $\Lambda$ is not contained in $\F_d^{ev}$, which concludes the proof. 
\end{proof}

\begin{cor}\label{cor.isomOEflexibility}
Let $\Lambda\leq \F_d$ be a finite index subgroup which is not included in the even subgroup $\F_d^{ev}$. Then the \pmp{} action $\F_d\curvearrowright (\Isom(\F_d)/\Lambda,m_\Lambda)$ and the diagonal action $\F_d\curvearrowright (\Isom(\F_d)/\F_d\times\F_d/\Lambda,m_{\F_d}\otimes u)$ are isometric orbit equivalent but the former is mixing, whereas the latter is not.
\end{cor}

\begin{rmq} Let $d\geq 0$ and $d'\geq 0$ be two natural numbers. Let 
\[\Gamma\coloneqq \langle x_1,\dots,x_d,y_1,\dots,y_{d'}\mid y_1^2=\dots=y_{d'}^2=1\rangle.\]
The group $\Gamma$ is isomorphic to the free product of the free group $\F_d$ and the universal Coxeter group $W_{d'}$ of rank $d'$. Let $S_\Gamma$ be the finite generating system given by 
\[S_\Gamma\coloneqq \{x_1^{\pm 1},\dots, x_d^{\pm 1}, y_1,\dots,y_{d'}\}.\]
Then the Cayley graph $(\Gamma,S_\Gamma)$ is isometric to a $(2d+d')$-regular tree. If $(2d+d')\geq 3$, then the same kind of arguments used in this section can be used to show that Corollary \ref{cor.isomOEflexibility} is true if $\F_d$ is replaced by $\Gamma$. 
\end{rmq}

\subsection{A concrete isometric orbit equivalence for $\F_2$}\label{sec.concreteexample}

We finish this section with a concrete example of two \pmp{} ergodic actions of $\F_2$ that are isometric orbit equivalent but not measurably isomorphic. 

Let $\F_2$ be the free group on two generators $a$ and $b$ and let $S=\{a^{\pm 1},b^{\pm 1}\}$ be the standard generating system. Let $\lvert\cdot\rvert_S$ be the word length associated with $S$. The Cayley graph $(\F_2,S)$ is isomorphic to the $4$-regular tree. Let $\Isom(\F_2)$ be the group of bijective isometries of $(\F_2,S)$, that is, the group of all bijections $f:\F_2\to\F_2$ such that for all $\gamma,\delta\in\F_2$, 
\[\lvert f(\gamma)\inv f(\delta)\rvert_S=\lvert \gamma\inv\delta\rvert_S.\]  Let $\mathcal{C}$ be the space of proper colorings with five colors on the vertex set of the Cayley graph $(\F_2,S)$. That is, an element of $\mathcal{C}$ is a map $col:\F_2\to\{1,2,3,4,5\}$ such that for all $\gamma\in \F_2$ and for all $s\in S$, we have $col(\gamma)\neq col(\gamma s)$. The set $\mathcal{C}$ is a closed, thus compact, subspace of $\{1,2,3,4,5\}^{\F_2}$. The group $\Isom(\F_2)$ acts on $\mathcal{C}$ and we denote by $*$ the action, which is defined as follows: for all $f\in\Isom(\F_2)$ and $c\in\mathcal{C}$, the coloring $f*col$ is given by $\gamma\mapsto col(f\inv(\gamma))$. This action is simply transitive and it admits a unique invariant probability measure $\mu$, which can be constructed as follows. First, choose uniformly at random the color of the identity element $1_{\F_2}$. Then, by moving radially outwards $1_{\F_2}$ in the Cayley graph, extend the coloring at each vertex by choosing the color uniformly at random among the admissible ones, independently at each vertex. Thus we get a \pmp{} action $*$ of $\F_2$ on $(\mathcal{C},\mu)$. Let us explain briefly why this action is mixing. Fix two distinct $5$-cycles $A,B\in\mathrm{Sym}(\{1,\dots,5\})$ such that for all $i\in\{1,\dots,5\}$, $A(i)\neq B(i)$. For instance, take
\[A \coloneqq (1\ 2\ 3\ 4\ 5) \text{ and }B\coloneqq(1\ 3\ 5\ 2\ 4).\]
This yields a transitive action $\F_2\curvearrowright \{1,2,3,4,5\}$ by letting the generator $a$ act like $A$ and $b$ act like $B$. Let $\Lambda$ be the stabilizer of the point $1$. If $m_\Lambda$ denotes the Haar probability measure on the quotient $\Isom(\F_2)/\Lambda$, then it can be proved that the \pmp{} action $*$ of $\F_2$ on $(C_2,\mu)$ is measurably isomorphic to $\F_2\curvearrowright (\Isom(\F_2)/\Lambda,m_\Lambda)$. By definition, $\F_2\curvearrowright\{1,\dots,5\}$ is isomorphic to the action $\F_2\curvearrowright\F_2/\Lambda$. As there is no bipartition of $\{1,\dots,5\}$ whose pieces are exchanged by any element $s\in S$, we get by Lemma \ref{lem.even} and Theorem \ref{thm.notcontained} that the \pmp{} action $*$ of $\F_2$ on $(\mathcal{C},\mu)$ is mixing. 

Let us construct another \pmp{} action of $\F_2$ on $(\mathcal{C},\mu)$, that we denote by $\star$, which is isometric orbit equivalent to 
$*$, but which is not mixing.

\begin{figure}[h!]
\centering
\begin{tikzpicture}[scale=1]
\foreach \i in {1,...,5}{
    \node[draw, circle,inner sep=1pt] (5\i) at (\i*360/5+360/20:2.2) {\tiny\i};}
    \draw[thick] (51) -- (52) node[
    currarrow,
    pos=0.5, 
    xscale=-1,
    sloped,
    scale=0.6] {};
    \draw[thick] (52) -- (53)node[
    currarrow,
    pos=0.5, 
    xscale=1,
    sloped,
    scale=0.5] {};
	\draw[thick] (53) -- (54)node[
    currarrow,
    pos=0.5, 
    xscale=1,
    sloped,
    scale=0.5] {};
	\draw[thick] (54) -- (55)node[
    currarrow,
    pos=0.5, 
    xscale=1,
    sloped,
    scale=0.5] {};
	\draw[thick] (55) -- (51)node[
    currarrow,
    pos=0.5, 
    xscale=-1,
    sloped,
    scale=0.5] {};
	\draw[color=red, thick] (51) -- (53)node[
    currarrow,
    pos=0.2, 
    xscale=-1,
    sloped,
    scale=0.5] {};
	\draw[color=red, thick] (53) -- (55) node[
    currarrow,
    pos=0.2, 
    xscale=1,
    sloped,
    scale=0.5] {};
	\draw[color=red, thick] (55) -- (52) node[
    currarrow,
    pos=0.2, 
    xscale=-1,
    sloped,
    scale=0.5] {};
	\draw[color=red, thick] (52) -- (54) node[
    currarrow,
    pos=0.2, 
    xscale=1,
    sloped,
    scale=0.5] {};
	\draw[color=red, thick] (54) -- (51) node[
    currarrow,
    pos=0.2, 
    xscale=-1,
    sloped,
    scale=0.5] {};


\node[draw, circle, inner sep=1pt] (e) at (8,0) {\tiny 2};
\node[draw, circle, inner sep=1pt] (a) at (10,0) {\tiny 1};
\node[draw, circle, inner sep=1pt] (A) at (6,0) {\tiny 5};
\node[draw, circle, inner sep=1pt] (b) at (8,2) {\tiny 4};
\node[draw, circle, inner sep=1pt] (B) at (8,-2) {\tiny 3};

\node[draw, circle, inner sep=1pt] (ab) at (10,1) {\tiny 5};
\node[draw, circle, inner sep=1pt] (aB) at (10,-1) {\tiny 4};
\node[draw, circle, inner sep=1pt] (aa) at (11,0) {\tiny 3};

\node[draw, circle, inner sep=1pt] (bb) at (8,3) {\tiny 1};
\node[draw, circle, inner sep=1pt] (bA) at (7,2) {\tiny 3};
\node[draw, circle, inner sep=1pt] (ba) at (9,2) {\tiny 5};

\node[draw, circle, inner sep=1pt] (AA) at (5,0) {\tiny 3};
\node[draw, circle, inner sep=1pt] (Ab) at (6,1) {\tiny 1};
\node[draw, circle, inner sep=1pt] (AB) at (6,-1) {\tiny 4};

\node[draw, circle, inner sep=1pt] (BB) at (8,-3) {\tiny 1};
\node[draw, circle, inner sep=1pt] (Ba) at (9,-2) {\tiny 5};
\node[draw, circle, inner sep=1pt] (BA) at (7,-2) {\tiny 4};

\draw[thick] (AA) -- (A) -- (e) -- (a) -- (aa); 
\draw[thick] (BB) -- (B) -- (e) -- (b) -- (bb); 
\draw[thick] (BA) -- (B) -- (Ba); 
\draw[thick] (bA) -- (b) -- (ba); 
\draw[thick] (aB) -- (a) -- (ab); 
\draw[thick] (AB) -- (A) -- (Ab); 

\draw[dotted, thick] (aa) --(11.4,0);
\draw[dotted, thick] (aa) --(11,0.4);
\draw[dotted, thick] (aa) --(11,-0.4);

\draw[dotted, thick] (AA) --(4.6,0);
\draw[dotted, thick] (AA) --(5,0.4);
\draw[dotted, thick] (AA) --(5,-0.4);

\draw[dotted, thick] (bb) --(7.6,3);
\draw[dotted, thick] (bb) --(8.4,3);
\draw[dotted, thick] (bb) --(8,3.4);

\draw[dotted, thick] (BB) --(7.6,-3);
\draw[dotted, thick] (BB) --(8.4,-3);
\draw[dotted, thick] (BB) --(8,-3.4);

\draw[dotted, thick] (ab) --(10,1.4);
\draw[dotted, thick] (ab) --(10.4,1);
\draw[dotted, thick] (ab) --(9.6,1);

\draw[dotted, thick] (aB) --(10,-1.4);
\draw[dotted, thick] (aB) --(10.4,-1);
\draw[dotted, thick] (aB) --(9.6,-1);

\draw[dotted, thick] (Ab) --(6,1.4);
\draw[dotted, thick] (Ab) --(6.4,1);
\draw[dotted, thick] (Ab) --(5.6,1);

\draw[dotted, thick] (AB) --(6,-1.4);
\draw[dotted, thick] (AB) --(6.4,-1);
\draw[dotted, thick] (AB) --(5.6,-1);

\draw[dotted, thick] (ba) --(9.4,2);
\draw[dotted, thick] (ba) --(9,2.4);
\draw[dotted, thick] (ba) --(9,1.6);

\draw[dotted, thick] (bA) --(6.6,2);
\draw[dotted, thick] (bA) --(7,2.4);
\draw[dotted, thick] (bA) --(7,1.6);

\draw[dotted, thick] (Ba) --(9.4,-2);
\draw[dotted, thick] (Ba) --(9,-2.4);
\draw[dotted, thick] (Ba) --(9,-1.6);

\draw[dotted, thick] (BA) --(6.6,-2);
\draw[dotted, thick] (BA) --(7,-2.4);
\draw[dotted, thick] (BA) --(7,-1.6);

\draw (e) node[below right]{\tiny$1_{\F_2}$};
\draw (a) node[above left]{\tiny$a$};
\draw (aa) node[below right]{\tiny$a^2$};
\draw (aB) node[below right]{\tiny$ab\inv$};
\draw (ab) node[above right]{\tiny$ab$};
\draw (AA) node[below left]{\tiny$a^{-2}$};
\draw (Ab) node[above left]{\tiny$a\inv b$};
\draw (AB) node[below left]{\tiny$a\inv b\inv$};
\draw (A) node[below right]{\tiny$a\inv$};
\draw (B) node[above right]{\tiny$b\inv$};
\draw (Ba) node[below right]{\tiny$b\inv a$};
\draw (BB) node[below right]{\tiny$b^{-2}$};
\draw (BA) node[below left]{\tiny$b\inv a\inv$};
\draw (bA) node[above left]{\tiny$b a\inv$};
\draw (bb) node[above left]{\tiny$b^2$};
\draw (ba) node[above right]{\tiny$ba$};
\draw (b) node[below left]{\tiny$b$};
\end{tikzpicture}

\caption{On the left, black arrows represent the permutation $A$, red arrows $B$. On the right, the portion of an element $c\in\mathcal{C}$, for which $a\star col=b *col$, $a\inv\star col=a\inv *col$, $b\star col=b\inv *col$ and $b\inv\star col=a *col$.}
\label{fig.isomOE}
\end{figure}
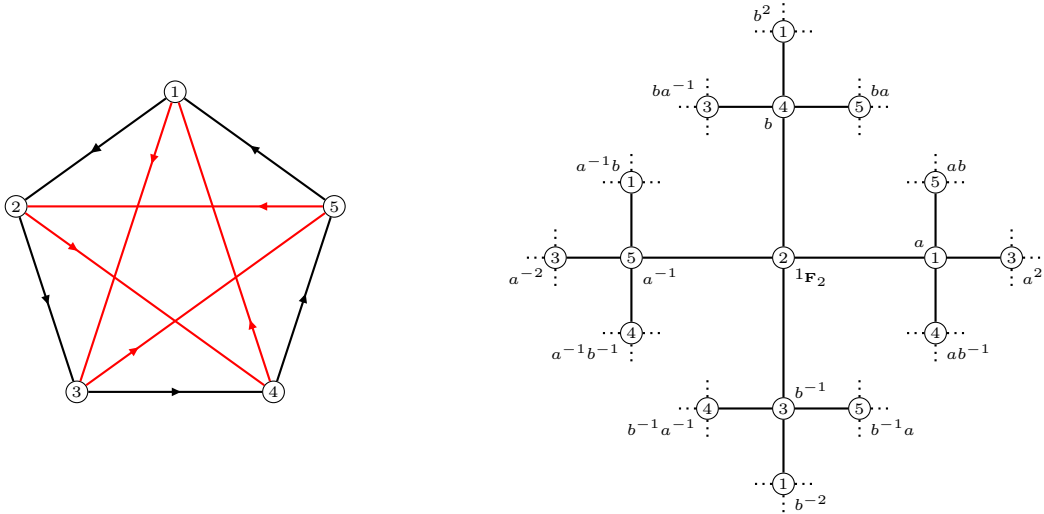

We define $\star$ by the action of the generators $a$ and $b$ of $\F_2$. For $col\in\mathcal{C}$, we let 
\[a\star col\coloneqq s*col,\]
where $s$ is the unique element in $\{a^{\pm 1},b^{\pm 1}\}$ such that $(s*col)(1_{\F_2})=A(col(1_{\F_2}))$. Such an element $s$ exists because $c$ is a proper vertex coloring. Similarly, we define 
\[b\star col\coloneqq t*col,\]
where $t$ is the unique element in $\{a^{\pm 1},b^{\pm 1}\}$ such that $(t*col)(1_{\F_2})=B(col(1_{\F_2}))$.

This defines a \pmp{} action $\star$ of $\F_2$ on $(C,\mu)$. Observe that for each $i\in\{1,\dots,5\}$, the set
\[\{col\in\mathcal{C}\colon col(1_\Gamma)=i\}\]
is invariant by $\Lambda$. Thus, the \pmp{} action $\star$ is not mixing. Moreover, by construction, the actions $*$ and $\star$ are isometric orbit equivalent. This yields a concrete illustration of Theorem \ref{thm.constructionisomOE}. Indeed, it can be showed that:

\begin{itemize}
\item[--] the \pmp{} action $*$ is measurably isomorphic to $\F_2\curvearrowright (\Isom(\F_2)/\Lambda,m_\Lambda)$, where $m_\Lambda$ is the Haar probability measure on $\Isom(\F_2)/\Lambda$,
\item[--] the \pmp{} action $\star$ is measurably isomorphic to the diagonal action $\F_2\curvearrowright (\Isom(\F_2)/\F_2\times \F_2/\Lambda, m_{\F_2}\otimes u)$ where $m_{\F_2}$ is the Haar probability measure on $\Isom(\F_2)/\F_2$ and $u$ is the uniform probability measure on $\F_2/\Lambda$. 
\end{itemize}

\appendix
\section{Free groups have the Haagerup property}\label{app}

In this appendix we provide a new proof that free groups have the Haagerup property. A countable group $\Gamma$ have the \defin{Haagerup property} if $\Gamma$ admits a proper continuous affine isometric action on a Hilbert space. It turns out that the Haagerup property can be characterized in terms of the existence of \pmp{} actions with suitable properties. A \pmp{} action $\Gamma\curvearrowright (X,\mu)$ of a countable group $\Gamma$ is \defin{strongly ergodic} if for all sequence $(A_n)_{n\geq 0}$ of measurable sets such that 
\[\underset{n\to +\infty}{\lim} \mu(A_n\triangle \gamma A_n), \forall \gamma\in\Gamma,\]
we have $\mu(A_n)(1-\mu(A_n))\longrightarrow0$. Informally, any sequence of measurable sets which is asymptotically invariant must be asymptotically null or conull. Jolissaint proved in \cite[Thm.~2.1.3]{Jolissaint} that a countable group $\Gamma$ satisfies the Haagerup property if and only if it admits a \pmp{} action on a standard probability space which is mixing but not strongly ergodic. For free groups, concrete examples of pmp mixing, non strongly ergodic actions can be constructed following the lines of the proof of \cite[Thm.~2.1.3]{Jolissaint} using Gaussian construction. Other concrete examples of such actions using Ising model with small transition probability are constructed in \cite[p.~90-91]{Hjorth}. We provide here new examples of such actions.


\begin{thm} Let $d\geq 2$. Let $\mathbf{F}_d=\Gamma_0\leq\Gamma_1\leq\dots\leq\Gamma_n\leq\dots$ be a sequence of finite index subgroups ,each normal in $\mathbf{F}_d$, such that $(i)$ the sequence of Schreier graphs $\mathrm{Sch}(\mathbf{F}_d/\Gamma_n,S_{\mathbf{F}_d})$ is not a family of expanders and $(ii)$ none of the Schreier graphs $\mathrm{Sch}(\mathbf{F}_d/\Gamma_n,S_{\mathbf{F}_d})$ is bipartite. Let $m_n$ denotes the Haar probability measure on $\Isom(\F_d)/\Gamma_n$. Then the projective limit
\[\mathbb{F}_d\curvearrowright \varprojlim(\Isom(\F_d)/\Gamma_n,m_n)\]
is a \pmp{} mixing action that is not strongly ergodic. 
\end{thm}

\begin{proof} Fix a sequence of finite index subgroups $\mathbf{F}_d=\Gamma_0\leq\Gamma_1\leq\dots\leq\Gamma_n\leq\dots$ as in the theorem. Since none of the Schreier graphs $\mathrm{Sch}(\mathbf{F}_d/\Gamma_n,S_{\F_d})$ is bipartite, we obtain by Theorem \ref{thm.notcontained} and Lemma \ref{lem.even} that the \pmp{} action $\F_d\curvearrowright (\Isom(\F_d)/\Gamma_n,\mu_n)$ is mixing. Since an inverse limit of mixing actions is mixing \cite[Thm.~7]{Brown}, we obtain that the projective limit 
\[\mathbb{F}_d\curvearrowright \varprojlim(\Isom(\F_d)/\Gamma_n,m_n)\]
is a \pmp{} mixing action. 

Let us now prove that this action is not strongly ergodic. For all $n\geq 0$, let $u_n$ denote the uniform probability measure on $\mathbf{F}_d/\Gamma_n$. Since the sequence of Schreier graphs $\mathrm{Sch}(\mathbf{F}_d/\Gamma_n,S_{\F_d})$ is not a family of expanders, the profinite action $\F_d\curvearrowright \varprojlim (\F_d/\Gamma_n,u_n)$ has no spectral gap. By \cite[Thm.~4]{AbertElekprofinite}, this profinite action is not strongly ergodic. By Corollary \ref{cor.isomOEflexibility}, the projective limit 
\[\mathbb{F}_d\curvearrowright \varprojlim(\Isom(\F_d)/\Gamma_n,m_n)\]
is (isometric) orbit equivalent to the diagonal action 
\[\F_d\curvearrowright (\Isom(\F_d)/\F_d,m_0)\times \varprojlim (\F_d/\Gamma_n,u_n),\]
which is not strongly ergodic. Since strong ergodicity is invariant under orbit equivalence, we obtain that the projective limit $\mathbb{F}_d\curvearrowright \varprojlim(\Isom(\F_d)/\Gamma_n,m_n)$ is not strongly ergodic, which concludes the proof.

\end{proof}

	\bibliographystyle{alpha}
	\bibliography{biblio}
	
	\Addresses
	
\end{document}